\documentclass[12pt]{amsart}
\usepackage{graphicx}
\usepackage[centertags]{amsmath}
\usepackage{amsfonts}
\usepackage{amssymb}
\usepackage{amsthm}

\newtheorem{thm}{Theorem}
\newtheorem*{thm*}{Theorem}
\newtheorem{lem}{Lemma}[section]
\newtheorem{cor}{Corollary}[section]
\newtheorem*{cor*}{Corollary}

\newtheorem{prop}[lem]{Proposition}
\theoremstyle{definition}

\theoremstyle{remark}
\newtheorem{rem}{Remark}[section]
\numberwithin{equation}{section}

\newcommand{\set}[1]{\left\{#1\right\}}

\newcommand{\smatrix}[4]{\left(\begin{smallmatrix} #1& #2\\ #3& #4\end{smallmatrix}\right)}
\newcommand{\eps}{\varepsilon}

\newcommand{\vphi}{{\varphi}}

\newcommand{\calA}{\mathcal{A}}
\newcommand{\calC}{\mathcal{C}}

\newcommand{\D}{\mathrm{D}}
\newcommand{\DL}{\mathrm{D}_{\mathrm{L}}}

\newcommand{\calH}{\mathcal{H}}

\newcommand{\calM}{\mathcal{M}}

\newcommand{\calO}{\mathcal{O}}

\newcommand{\bbZ}{\mathbb{Z}}

\newcommand{\bbQ}{\mathbb{Q}}
\newcommand{\bbR}{\mathbb R}
\newcommand{\bbC}{\mathbb C}

\newcommand{\bbF}{\mathbb F}
\newcommand{\bbN}{\mathbb N}

\newcommand{\bbH}{\mathbb{H}}

\newcommand{\fraka}{\mathfrak{a}}

\newcommand{\frakg}{\mathfrak{g}}
\newcommand{\frakk}{\mathfrak{k}}
\newcommand{\frakn}{\mathfrak{n}}
\newcommand{\frakm}{\mathfrak{m}}
\newcommand{\frakq}{\mathfrak{q}}
\newcommand{\frakp}{\mathfrak{p}}
\newcommand{\frakP}{\mathfrak{P}}

\newcommand{\ls}{\textsc{ls}}
\newcommand{\pls}{\textsc{ls}^o}
\newcommand{\s}{\textsc{s}}
\newcommand{\m}{\mathrm{m}}
\newcommand{\Tr}{ \mbox{Tr}}

\newcommand{\SL}{ \mathrm{SL}}
\newcommand{\SO}{ \mathrm{SO}}
\newcommand{\Sl}{ \mathrm{sl}}
\newcommand{\PSL}{ \mathrm{PSL}}

\newcommand{\SU}{ \mathrm{SU}}

\newcommand{\GL}{ \mbox{GL}}

\newcommand{\Sp}{\mathrm{Sp}}

\newcommand{\Mat}{\mathrm{Mat}}
\newcommand{\Gal}{\mathrm{Gal}}

\newcommand{\vol}{\mathrm{vol}}

\newcommand{\bs}{\backslash}
\newcommand{\id}{1\!\!1}

\newcommand{\Ad}{\mathrm{Ad}}
\newcommand{\ad}{\mathrm{ad}}

\newcommand{\sgn}{\mathrm{sgn}}

\begin{document}
\title[Refinement of multiplicity one]
{A refinement of strong multiplicity one for spectra of hyperbolic manifolds}%
\author{Dubi Kelmer}%
\address{Department of Mathematics, 301 Carney Hall, Boston College
Chestnut Hill, MA 02467}
\email{dubi.kelmer@bc.edu}

\thanks{}%

\date{\today}%
\dedicatory{}%
\commby{}%

\maketitle

\begin{abstract}
Let $\calM_1$ and $\calM_2$ denote two compact hyperbolic manifolds. Assume that the multiplicities of eigenvalues of the Laplacian acting on $L^2(\calM_1)$ and $L^2(\calM_2)$ (respectively, multiplicities of lengths of closed geodesics in $\calM_1$ and $\calM_2$) are the same, except for a possibly infinite exceptional set of eigenvalues (respectively lengths). We define a notion of density for the exceptional set and show that if it is below a certain threshold, the two manifolds must be iso-spectral.
\end{abstract}

\section*{introduction}
By a hyperbolic manifold in this paper we refer to a quotient of a real, complex, quaternionic, or octonionic hyperbolic space.
Specifically, let $G$ denote a connected semisimple Lie group of real rank one, $K\subset G$ a maximal
compact subgroup and $\calH=G/K$ the corresponding symmetric space endowed with the hyperbolic Riemannian metric coming from the
Killing form on $\frakg=\rm{Lie}(G)$. Then for any uniform torsion free lattice $\Gamma\subset G$, the locally symmetric space $\calM=\Gamma\bs \calH$ has the structure of a compact hyperbolic manifold. By the spectra of $\calM$, that we also call the spectra of $\Gamma$, we refer to any one of the following notions described below.

The \emph{Laplace spectrum} is the set of eigenvalues of the Laplace-Beltrami operator acting on $L^2(\calM)$ listed with multiplicities. The \emph{representation spectrum} is the set of $\pi\in \hat{G}$ occurring in the right regular representation of $G$ on $L^2(\Gamma\bs G)$ listed with multiplicities. Here $\hat{G}$ denotes the unitary dual of $G$, that is, the set of equivalence classes of unitary representations of $G$. We note that the representation spectrum determines the Laplace spectrum as the multiplicity of each Laplace eigenvalue is given by the multiplicity of a corresponding representation. We say that two lattices are Laplace equivalent (or iso-spectral) if they have the same Laplace spectrum, and we say that they are representation equivalent if they have the same representation spectrum.

The \emph{length spectrum} (respectively \emph{primitive length spectrum}) is the set of lengths of closed geodesics (respectively primitive closed geodesics) in $\calM$ listed with multiplicities. 
The \emph{complex length spectrum} is the set of pairs, length and holonomy class, of all closed geodesics listed with multiplicities, where the holonomy class of a closed geodesic is the conjugacy class in $\SO(d-1),\; d=\dim(\calM)$ obtained by parallel transporting tangent vectors along the geodesic.  
We say that two spaces are length equivalent if they have the same length spectrum, and that they are complex length equivalent if they have the same complex length spectrum.

When $\calM$ is a hyperbolic surface, Huber \cite{Huber59} showed that the Laplace spectrum and the length spectrum determine each other (and they also determine the representation and complex length spectra). In higher dimensions the complex length spectrum and the representation spectrum still determine each other (see e.g. \cite{Salvai02}), however, the relation between the Laplace spectrum and the length spectrum is more mysterious.

In \cite{Gangolli77}, Gangolli generalized Huber's result to all hyperbolic manifolds, however, this generalization involves a different notion of length spectrum, we will call the $1$-length spectrum, assigning to each length a certain real positive number (see remark \ref{r:sigmalength}).
His result implies in particular that the Laplace spectrum of a hyperbolic manifold determines the length set (this is in fact true for all negatively curved manifolds \cite{DuistermaatGuillemin75}). However, the question of whether it determines the multiplicities, and the converse question of whether the length spectrum determines the Laplace spectrum were left open.

In \cite{BhagwatRajan10},  Bhagwat and Rajan showed that if two lattices satisfy that all but finitely many Laplace eigenvalues (respectively representations) have the same multiplicities, then they are Laplace equivalent (respectively, representation equivalent\footnote{The result on the representation spectrum is proved for any semisimple group not necessarily of rank one.}). In \cite{BhagwatRajan11} they studied the length spectrum of real hyperbolic manifolds of even dimension, showing that if two lattices have the same multiplicities for all but finitely many lengths, then they are length equivalent.
Similar results were previously shown by Elstrodt, Grunewald, and Mennicke in \cite[Theorem 3.3]{ElstrodtGrunewaldMennicke98} for the Laplace spectrum and $1$-length spectrum of hyperbolic $3$-manifolds.  Elstrodt, Grunewald, and Mennicke then asked if it is possible to prove the same result when the finite set is replaced with an infinite set of sufficiently small density in some suitable sense.

In this paper we give a positive answer to the question of Elstrodt, Grunewald, and Mennicke, for the Laplace spectrum, representation spectrum, $1$-length spectrum, and  length spectrum.
On the way, we also answer a question of Bhagwat and Rajan \cite{BhagwatRajan11} and show that the Laplace spectrum of any compact hyperbolic manifold is completely determined by its length spectrum (cf. \cite{Dietmar89} for a similar result for the spectrum of the Laplacian on forms). The second problem of whether the Laplace spectrum determines the multiplicities in the length spectrum remains open.

\begin{rem}
The results of Bhagwat and Rajan can be thought of as an analogue of the strong multiplicity one theorem for cusp forms. This theorem states that if $f$ and $g$ are two Hecke new-forms for which the eigenvalues of the Hecke operators are equal at all but finitely many primes, then they are equal at all primes, and $f=g$; see \cite[p. 125]{Lang76}. (The analogy is only with the first part of the theorem, as there are examples of iso-spectral but not isometric hyperbolic manifolds; cf. \cite{Sunada85,Vigneras80}.)
Continuing with this analogy, our result can be compared to Ramakrishnan's refinement \cite{Ramakrishnan94}, stating that the finite set of primes in the strong multiplicity one theorem can be replaced with an infinite set, as long as it has Dirichlet density less than $1/8$.
\end{rem}

In order to describe our results we need to introduce some more notation.
For any uniform lattice $\Gamma\subset G$ and any $\pi\in \hat{G}$ we denote by $\m_\Gamma(\pi)$ the multiplicity of $\pi$ in $L^2(\Gamma\bs G)$.
For every $\ell\in (0,\infty)$ we denote by $\m_{\Gamma}(\ell)$ (respectively $\m^o_{\Gamma}(\ell)$) the number of closed geodesics (respectively primitive closed geodesics) of length $\ell$. We also denote by $\ls_\Gamma$ (respectively $\pls_\Gamma$) the set of lengths of (primitive) closed geodesics in $\calM=\Gamma\bs \calH$.
Let $G=NAK$ be an Iwasawa decomposition of $G$ and let $M$ denote the centralizer of $A$ in $K$. For any $\sigma\in \hat{M}$ let $\pi_{\sigma,\nu},\;\nu\in i\bbR$ and $\pi_{\sigma,\nu},\;\nu\in (0,\rho)$ denote the unitary principle and complementary series in $\hat{G}$ where $\rho=\rho_G$ denotes half the sum of the positive roots for $(G,A)$ (see section \ref{s:BasicStructure} for more details).

For any two lattices $\Gamma_1,\Gamma_2\subseteq G$ and every $\sigma\in \hat{M}$ we define the following spectral density function to measure the difference between the (principal part) of the representation spectrum.
\begin{equation}\label{e:Dsigma}
\D_\sigma(\Gamma_1,\Gamma_2;T)=\sum_{\nu\in i\bbR,\;|\nu|\leq T}|\m_{\Gamma_1}(\pi_{\sigma,\nu})-\m_{\Gamma_2}(\pi_{\sigma,\nu})|.
\end{equation}
We think of the asymptotic growth rate of these functions as $T\to\infty$ as describing the ``density" of places for which the multiplicities are different.
We note that this captures more information as it also takes into account by how much the multiplicities differ.
Using this notion we prove a refinement of \cite[Theorems 1.1 and 1.2]{BhagwatRajan10} for real rank one groups.

\begin{thm}\label{t:Spec2Spec}
Let $G$ denote a real rank one group and $\Gamma_1,\Gamma_2\subseteq G$ two uniform lattices without torsion.
\begin{enumerate}
\item Let $1\in \hat{M}$ denote the trivial representation. If
$$\lim_{T\to\infty}\frac{\D_1(\Gamma_1,\Gamma_2;T)}{T}=0,$$
then $\Gamma_1$ and $\Gamma_2$ are Laplace equivalent.
\item Let $\mathfrak{B}\subseteq \hat{M}$ be a finite set.
If
$$\lim_{T\to\infty}\frac{\D_\sigma(\Gamma_1,\Gamma_2;T)}{T}=0,\quad \forall\sigma\not\in \mathfrak{B},$$
then $\Gamma_1$ and $\Gamma_2$ are representation equivalent.
\end{enumerate}
\end{thm}
\begin{rem}
The finite set $\mathfrak{B}$ above can also be replaced by an infinite set satisfying a certain sparsity condition. Moreover, when $M=\SO(2),\SO(3)$ or $\SU(2)$ we can naturally identify $\hat{M}$ with $\bbN$ and this condition is then essentially that $\mathfrak{B}$ is of density zero (see section \ref{s:FurtherRefinement}).
\end{rem}

For any two lattices $\Gamma_1,\Gamma_2\subseteq G$ we also define a length density function
\begin{equation}\label{e:Dls}\DL(\Gamma_1,\Gamma_2;T)=\sum_{\ell\leq T}|\m_{\Gamma_1}^o(\ell)-\m_{\Gamma_2}^o(\ell)|,\end{equation}
where the sum is over $\ell\in \pls_{\Gamma_1}\cup\pls_{\Gamma_2}$.
Our first results on the length spectrum is of a slightly different nature than the results in \cite{BhagwatRajan10,BhagwatRajan11}, in the sense that we start with (partial data) on the length spectrum and retrieve the Laplace spectrum.
\begin{thm}\label{t:Length2Laplace}
Let $G$ denote a real rank one group and $\Gamma_1,\Gamma_2\subseteq G$ two uniform lattices without torsion.
If $\mathrm{D}_{\rm{L}}(\Gamma_1,\Gamma_2;T)\lesssim e^{\alpha T}$
with $\alpha<\rho_G$, then $\Gamma_1$ and $\Gamma_2$ are Laplace equivalent.
\end{thm}
\begin{cor*}
If two compact hyperbolic manifolds are length equivalent then they are Laplace equivalent.
\end{cor*}



For hyperbolic surfaces the Laplace spectrum determines the length spectrum, hence, for surfaces the threshold in Theorem \ref{t:Length2Laplace} also implies that the two lattices are length equivalent.
\begin{rem}
It is interesting to compare this to the result of Buser \cite[Theorem 10.1.4]{Buser92} showing that there is a constant $c(g,\epsilon)$ such that if two hyperbolic surfaces of genus $g$ and injectivity radius $\geq \epsilon$ have the same multiplicities for all lengths $\leq c(g,\epsilon)$, then they are length equivalent.
\end{rem}

In higher dimensions, we do not know if the Laplace spectrum determines the length spectrum.
Nevertheless, with the exception of the odd dimensional real hyperbolic spaces, if we impose a smaller threshold for the growth rate we are able to recover the length spectrum directly. Specifically, for each of the rank one groups we define the threshold
\begin{equation}\label{e:threshold}\alpha_0(G)=\left\lbrace\begin{array}{ll}
0 & G=\SO_0(2n+1,1),\;n\in \bbN\\
1/2 & G=\SO_0(2n+2,1),\;n\in \bbN\\
1 & \mbox{ otherwise}
\end{array}\right.\end{equation}
With this threshold we can prove a refinement of \cite[Theorem 1]{BhagwatRajan11}.
\begin{thm}\label{t:Length2Length}
Let $G$ denote a real rank one group and $\Gamma_1,\Gamma_2\subseteq G$ two uniform lattices without torsion.
If $\DL(\Gamma_1,\Gamma_2;T)\lesssim e^{\alpha T}$
for some $\alpha<\alpha_0(G)$, then $\Gamma_1$ and $\Gamma_2$ are length equivalent.
\end{thm}
For odd dimensional real hyperbolic space the threshold $\alpha_0=0$ so the statement is empty.
In this case, even assuming that $\m_{\Gamma_1}^o(\ell)=\m_{\Gamma_2}^o(\ell)$ for all but finitely many values of $\ell$ we were not able to prove that $\Gamma_1$ and $\Gamma_2$ are length equivalent. However, from Theorem \ref{t:Length2Laplace} we know that they must be
Laplace equivalent and hence must have the same length set and the same volume. Using this fact we can show
\begin{thm}\label{t:Length2LengthOdd}
Let $\Gamma_1,\Gamma_2\subseteq \SO_0(2n+1,1)$ denote two uniform lattices without torsion. If $\m^o_{\Gamma_1}(\ell)=\m^o_{\Gamma_2}(\ell)$ for all  $\ell\not\in \{\ell_1,\ldots\ell_k\}$ then
$\ell_1,\ldots,\ell_k$ are rationally dependent. In particular, if $k=1$ then $\Gamma_1$ and $\Gamma_2$ are length equivalent.
\end{thm}

We now discuss the sharpness of our thresholds for the density functions.
Regarding Theorems \ref{t:Length2Length} and \ref{t:Length2LengthOdd}, we note that there are no known examples of hyperbolic manifolds that are Laplace equivalent but not representation equivalent. It is thus possible that the correct threshold is actually the same as in Theorem \ref{t:Length2Laplace}.

For Theorem \ref{t:Length2Laplace} we recall the Prime Geodesic Theorem \cite{Gangolli77,Margulis69}, stating that
\[\sum_{\ell\leq T}\m_{\Gamma}^o(\ell) \sim \rm{Li}(e^{2\rho T})\sim  \frac{e^{2\rho T}}{2\rho T}.\]
Our threshold is thus roughly the square root of the trivial bound $\DL(\Gamma_1,\Gamma_2;T)\leq 2\rm{Li}(e^{2\rho T})$.
We note that this is analogous to the threshold in the result of Soundararajan \cite{Soundararajan04} on strong multiplicity one for the Selberg class.
On the other hand, the analogy with the density $1/8$ result for cusp forms may lead to the suspicion that the correct threshold should be of the form $c\rm{Li}(e^{2\rho T})$ with some $c<2$ a sufficiently small constant.
Though we do not know if our threshold of $e^{\rho T}$ is sharp, we show that such a positive density threshold cannot hold. Specifically, borrowing examples from \cite{LeiningerMcReynoldsNeumannReid07} we show
\begin{thm}\label{t:NoPositiveDensity}
For every $\epsilon>0$ there are non iso-spectral $\Gamma_1,\Gamma_2\subset G$ satisfying
\[\limsup_{T\to\infty}\frac{\DL(\Gamma_1,\Gamma_2;T)}{\rm{Li}(e^{2\rho T})}\leq \epsilon.\]
\end{thm}

\begin{rem}
We note that the co-volumes of the lattices we use in the proof go to infinity as $\epsilon\to 0$.
It is thus still possible that a positive density threshold can hold under the additional assumption that the volumes are uniformly bounded.
See section \ref{s:Perlis} for a similar phenomenon in the analogous context of arithmetically equivalent number fields.
\end{rem}

For the representation spectrum we suspect that our threshold is not optimal.
We recall that the Weyl law for the principal spectrum is
\[\sum_{|\nu|\leq T} \m_\Gamma(\pi_{\sigma,\nu})\sim C\dim(\sigma)\vol(\Gamma\bs G)T^d,\]
with $d=\dim(G/K)$ and $C$ an explicit constant depending on $G$ (see \cite{MiatelloVargas83}). Consequently, the trivial bound for $\D_{\sigma}(\Gamma_1,\Gamma_2;T)$ is $O(T^d)$ and the correct threshold could very well be
$$\lim_{T\to\infty}\frac{\D_\sigma(\Gamma_1,\Gamma_2;T)}{T^{d}}=0,$$
or even a positive density threshold of the form $D_\sigma(\Gamma_1,\Gamma_2;T)\leq cT^d$ with $c$ a sufficiently small constant.
We note that the first condition implies the two lattices at least have the same co-volume.

We conclude this introduction with a brief outline of the paper. In Section \ref{s:Background} we introduce some notation and recall some basic results on the spectral theory of symmetric spaces. In section \ref{s:Spec} we give the proof of Theorem \ref{t:Spec2Spec}
using the Selberg trace formula. Our proof is similar to the original proof of \cite[Theorem 3.3]{ElstrodtGrunewaldMennicke98};
the new ingredient which allows us to improve on their result is the use of more general test functions in the trace formula (instead of just the heat trace).
In section \ref{s:Alternating} we develop a new trace formula in which the geometric side involves the length spectrum directly.
The price we have to pay is that on the spectral side, in addition to the Laplace spectrum, we have contribution from other representations occurring in $L^2(\Gamma\bs G)$. In section \ref{s:Length} we show that, by using suitable test functions in the trace formula, we can isolate the contribution of each of those representations. This enables us to prove Theorems \ref{t:Length2Laplace},\ref{t:Length2Length}, and \ref{t:Length2LengthOdd}.
Finally, in section \ref{s:Similar} we recall the construction in \cite{LeiningerMcReynoldsNeumannReid07} of lattices with the same length sets and use it to prove Theorem \ref{t:NoPositiveDensity}.

\subsection*{Acknowledgments}
I thank Ben McReynolds and Alan Reid for explaining their construction of spaces with the same length sets.
I also thank Peter Sarnak and Masato Wakayama for clarifying a few points regarding the trace formula.
This work was partially supported by NSF grant DMS-1001640.

\section{Notation and preliminaries}\label{s:Background}
We write $A\lesssim B$ or $A=O(B)$ to indicate that $A\leq cB$
for some constant $c$. We also write $A\asymp B$ to indicate that
$A\lesssim B\lesssim A$. We write $A(T)\sim B(T)$ if $A(T)/B(T)\to 1$  and $A(T)=o(B(T))$ if $A(T)/B(T)\to 0$ as $T\to \infty$.

\subsection{Basic structure on symmetric spaces}\label{s:BasicStructure}
Let $G$ denote a connected semisimple Lie group of real rank one, $K\subset G$ a maximal compact subgroup and $\calH=G/K$ the corresponding symmetric space. That is $G=\SO_0(n+1,1),\;\SU(n+1,1),\;\Sp(n,1)$ or $FII$ and $\calH$ is real, complex, quaternionic, or octonionic hyperbolic space respectively.

Fix an Iwasawa decomposition $G=NAK$ and let $\frakg=\frakn\oplus \fraka\oplus \frakk$ denote the corresponding decomposition of the Lie algebra $\frakg$.
Let $M,M^*\subseteq K$ denote the centralizer and normalizer of $A$ in $K$ respectively.
Let $W=W(G,A)=M^*/M$ denote the baby Weyl group; since we assume that $\dim(\fraka)=1$ then $|W|=2$ and we write $W=\{1,w\}$.
We denote by $\Sigma=\Sigma(G,A)$ the set of restricted roots for the pair $(G,A)$ and by $\Sigma^+$ the set of positive restricted roots, then either $\Sigma^+=\{\alpha\}$ or $\Sigma^+=\{\alpha,2\alpha\}$. Let $\rho=\rho_G$ denote half the sum of the positive roots, that is,
$$\rho=\frac{(\dim(\frakn_1)+2\dim(\frakn_2))\alpha}{2}$$
where $\frakn=\frakn_1\oplus\frakn_2$ is the decomposition into the root spaces of $\alpha_1$ and $\alpha_2$ respectively.
We fix (once and for all) an element $H\in \fraka$ with $\alpha(H)=1$ and for any $t\in \bbR$ we denote by $a_t=\exp(tH)\in A$. We can identify the dual spaces $\fraka^*=\bbR$ and $\fraka_\bbC^*=\bbC$
via $\nu=\nu(H)$. With this identification $\rho=\frac{\dim(\frakn_1)+2\dim(\frakn_2)}{2}$.

\subsection{The unitary dual}
Let $\hat{G},\hat{K}$, and $\hat{M}$ denote the unitary duals of $G,K$, and $M$ respectively.
For any $\sigma\in \hat{M}$ we denote by $\pi_{\sigma,\nu},\;\nu\in i\bbR$ the principal series representations and by $\pi_{\sigma,\nu},\;\nu\in (0,\rho)$ the complementary series representations.

The action of $M^*$ on $M$ (by conjugation) induces an action of the Weyl group $W=\{1,w\}$ on $\hat{M}$. We note that under this action $\pi_{\sigma,i\nu}=\pi_{w\sigma,-i\nu}$ and that these are the only pairs of equivalent principal series representation. We say that a representation $\sigma\in \hat{M}$ is ramified if $\sigma=w\sigma$ and unramified otherwise, and we recall that there are unramified $\sigma\in \hat{M}$ only when 
$G=\SO_0(2m+1,1)$.

We denote by $\hat{G}_c$ the set of equivalence classes of the principle and complementary series representations, and by $\mathfrak{S}=\hat{G}\setminus \hat{G}_c$ (that is, $\mathfrak{S}$ is the set of equivalence classes of discrete series representation, limits of discrete series, and Langland's quotients).


For any $\pi\in \hat{G}$  we denote by $\chi_\pi$ the infinitesimal character of $\pi$.
We denote by $\Omega_G$ the Casimir operators of $G$. Since $G$ is of rank one the character $\chi_\pi$ is uniquely determined by its value on $\Omega_G$.
For $\pi=\pi_{\sigma,\nu}\in G_c$ we have that $\chi_{\sigma,\nu}(\Omega_G)=\nu^2-\rho^2+\chi_\sigma(\Omega_M)$ where $\Omega_M$ denotes the Casimir operator of $M$ (appropriately normalized).
For any $\Lambda\in \bbC$ we denote by $\hat{G}(\Lambda)$  the (finite set) of representations with $\chi_{\pi}(\Omega_G)=\Lambda$ and by $\mathfrak{S}(\Lambda)=\mathfrak{S}\cap \hat{G}(\Lambda)$.

\subsection{Closed geodesics and conjugacy classes}
For any $\gamma\in \Gamma$ we denote by $[\gamma]\in \Gamma^\#$ its conjugacy class. Since $\calM$ is of negative curvature, there is a natural correspondence between conjugacy classes in $\Gamma=\pi_1(\calM)$, free homotopy classes of closed curves in $\calM$, and (oriented) closed geodesics.

We say that an element $\gamma\in \Gamma$ is primitive if it cannot be written as $\gamma=\delta^j$ for some other $\delta\in \Gamma$; note that this only depends on the $\Gamma$-conjugacy class. To any $[\gamma]\in \Gamma^\#$, we define the primitivity index $j(\gamma)$ as the unique $j\in \bbN$ such that $\gamma=\delta^j$ with $\delta\in \Gamma$ primitive. Under the above correspondence, a closed geodesic is primitive if and only if the corresponding conjugacy class is primitive. Moreover, the primitivity index $j(\gamma)$ is the number of times the geodesic wraps around itself.

Any hyperbolic $\gamma\in \Gamma$ is conjugated in $G$ to an element $m_\gamma a_{\ell_\gamma}\in MA^+$ where $A^+=\{a_t|t>0\}$. Here $\ell_\gamma$ is uniquely determined by $[\gamma]$ and $m_\gamma$ is determined up to conjugacy in $M$. The pair $(\ell_\gamma,[m_{\gamma}])$ is then precisely the length and holonomy of the closed geodesic corresponding to $[\gamma]$.

\subsection{The $\sigma$-length and representation spectra}\label{s:sigmaspectra}
To any irreducible representation $\sigma\in \hat{M}$ we attach two function $L_{\Gamma,\sigma}:\bbR^+\to \bbR$ and $\m_{\Gamma,\sigma}:(0,\rho)\cup i\bbR\to \bbN$, we call the $\sigma$-length spectrum and $\sigma$-representation spectrum respectively. These two functions are closely related via the Selberg trace formula.

The $\sigma$-length spectrum is defined by
\[L_{\Gamma,\sigma}(\ell)=\mathop{\sum_{[\gamma]\in \Gamma^\#}}_{\ell_\gamma=
\ell}\frac{\overline{\chi_{\sigma}(m_{\gamma})}\ell}{2j(\gamma)D(\gamma)},\]
where $\chi_\sigma$ is the character of $\sigma$ and
\begin{equation}\label{e:WeylDiscriminant}D(\gamma)=e^{\rho\ell_\gamma}|\det((\Ad(m_\gamma a_{\ell_\gamma})^{-1}-I)_{|\frakn})|,\end{equation}
is the Weyl discriminant. When $\sigma$ is unramified we also define
$$L^\pm_{\Gamma,\sigma}(\ell)=L_{\Gamma,\sigma}(\ell)\pm L_{\Gamma,w\sigma}(\ell).$$
\begin{rem}\label{r:sigmalength}
The definition of the length spectrum of a hyperbolic 3-manifold given in \cite[Definition 3.1]{ElstrodtGrunewaldMennicke98}
coincides with what we call the $1$-length spectrum, that is, the $\sigma$-length spectrum for $\sigma=1$ the trivial representation.
\end{rem}

In order to define the $\sigma$-representation spectrum we fix a virtual representation $\eta=\sum^\oplus_{\tau\in \hat{K}} a_\tau \tau$ with  ($a_\tau\in \bbZ$ almost all zeros) such that $\eta_{|_M}=\sigma$ (respectively $\sigma+w\sigma$ if $\sigma$ is unramified).
Let $\Lambda_{\sigma,\nu}=\chi_{\sigma,\nu}(\Omega_G)$
and let $\mathfrak{S}=\hat{G}\setminus \hat{G}_c$. For $\pi\in \mathfrak{S}$ we denote by $\alpha_{\Gamma}(\pi)$ the corrected multiplicity given by
$$\alpha_{\Gamma}(\pi)=\left\lbrace\begin{array}{ll}
\m_\Gamma(\omega)-\vol(\Gamma\bs G)d_\omega & \pi=\omega \mbox{ in discrete series}\\
\m_\Gamma(\pi)& \mbox{ otherwise}\end{array}\right.,$$
with $d_\omega$ the formal degree of $\omega$.
 The $\sigma$-representation spectrum is given by
\begin{equation}\label{e:SigmaSpec}
\m_{\Gamma,\sigma}(\nu)=\m_{\Gamma}(\pi_{\sigma,\nu})+\sum_{\pi\in \mathfrak{S}(\Lambda_{\sigma,\nu})} \alpha_{\Gamma}(\pi)[\pi_{|_K};\eta],
\end{equation}
where $[\pi_{|_K};\eta]=\sum_\tau a_\tau [\pi_{|_K};\tau]$.
For $\sigma$ unramified we also define $\m_{\Gamma}^\pm(\nu)$ by
\begin{equation}\label{e:SigmaSpecRam}
\m^+_{\Gamma,\sigma}(\nu)=\m_{\Gamma}(\pi_{\sigma,\nu})+ \m_{\Gamma}(\pi_{w\sigma,\nu})+\sum_{\pi\in \mathfrak{S}(\Lambda_{\sigma,\nu})} \m_{\Gamma}(\pi)[\pi_{|_K};\eta].
\end{equation}
and
\begin{equation}
\m^-_{\Gamma,\sigma}(\nu)=\m_{\Gamma}(\pi_{\sigma,\nu})- \m_{\Gamma}(\pi_{w\sigma,\nu})
\end{equation}

We note that by \cite[Theorem 1.2]{Miatello82} the $\sigma$-representation spectrum does not depend on the choice of $\eta$.
Also, since representations $\pi\in \mathfrak{S}$ have a minimal $K$-type (see \cite[Chapter XV]{Knapp86}), then for any fixed $\eta$, there are only finitely many $\pi\in \mathfrak{S}$ for which $[\pi_{|_K};\eta]\neq 0$. In particular, for $\sigma\in \hat{M}$ fixed,
$\m_{\Gamma,\sigma}(\nu)=\m_{\Gamma}(\pi_{\sigma,\nu})$ for all but finitely many values of $\nu$.
Moreover, for $\sigma=1$ trivial, $\m_{\Gamma,1}(\nu)=\m_{\Gamma}(\pi_{1,\nu})$ for all $\nu$.
For any $\sigma\in \hat{M}$ we define the $\sigma$-spectral set as
$$S_{\Gamma,\sigma}=\{\nu\in i\bbR^+\cup (0,\rho)|\;\m_{\Gamma,\sigma}(\nu)\neq 0\}.$$

\subsection{Trace formula attached to $\sigma$}
Building on the Selberg trace formula developed in \cite{Wallach76,Warner79},
Sarnak and Wakayama \cite{SarnakWakayama99} derived a trace formula attached to each irreducible representation $\sigma\in \hat{M}$.
They derived this formula in general for a (possibly) nonuniform lattice. We will write it down only for the simpler case of a uniform lattice without torsion. The derivation in this case is much simpler as there are no contribution from continuous spectrum or unipotent elements and the treatment of the multiplicities of discrete series is straight forward.
\begin{rem}
For the case of a uniform lattice this formula, with a special test function coming from the fundamental solution to the heat equational, was already derived in \cite{Miatello82,MiatelloVargas83}. We note that, in addition to the treatment of non-uniform lattices, another new features in \cite{SarnakWakayama99} which is crucial for our application is the use of general test functions. See also \cite{Dietmar89} for a similar trace formula.
\end{rem}

For $\Gamma\subseteq G$ a uniform lattice without torsion the trace formula corresponding to $\sigma\in \hat{M}$ takes the following form (see \cite[Theorem 2 and Theorem 6.5]{SarnakWakayama99} \footnote{We corrected here a few typos from the formulas in \cite{SarnakWakayama99} and in particular the mistake in \cite[Theorem 6.5 (2)]{SarnakWakayama99} for an odd test function.}:
\begin{itemize}
\item For $\sigma\in \hat{M}$ ramified, for any \textbf{even} $g\in C^\infty_c(\bbR)$,
\begin{eqnarray}\label{e:sigmatrace}
\sum_{\nu_k\in S_{\Gamma,\sigma}} \m_{\Gamma,\sigma}(\nu_k)\hat{g}(i\nu_k)&=& \vol(\Gamma\bs G)\int_{\bbR}\hat{g}(\nu)\mu_{\sigma}(\nu)d\nu\\ \nonumber&+&\sum_{\ell\in \ls_{\Gamma}} g(\ell)L_{\Gamma,\sigma}(\ell)
\end{eqnarray}
where $\hat{g}$ denotes the Fourier transform of $g$ and $\mu_\sigma(\nu)d\nu$ is the Plancherel measure.
\item For $\sigma\in \hat{M}$ unramified, for any \textbf{even} $g\in C^\infty_c(\bbR)$ we have the same formula but with $\m_{\Gamma,\sigma}^+$,
$L^+_{\Gamma,\sigma}$ and $\mu_\sigma^+=2\mu_\sigma$. In addition, for any \textbf{odd} $g\in C^\infty_c(\bbR)$ we have
\begin{eqnarray}\label{e:sigmatraceodd}
\sum_{\nu_k\in \s_{\Gamma,\sigma}}\m^-_{\Gamma}(\pi_{\sigma,\nu_k})\hat{g}(i\nu_k)=\sum_{\ell\in \ls_{\Gamma}}g(\ell) L^-_{\Gamma,\sigma}(\ell).
\end{eqnarray}
\end{itemize}

\begin{rem}
For any virtual representation $\eta=\sum a_\sigma \sigma$ of $M$, we can define $\m_{\Gamma,\eta},\; L_{\Gamma,\eta}$ and $\mu_\eta$ as the corresponding weighted sums. With this convention, the above trace formula holds for any virtual representation and not just the irreducible representations.
\end{rem}


\section{Proof of Theorem \ref{t:Spec2Spec}}\label{s:Spec}
Let $G$ denote a fixed semisimple group of real rank one and $\Gamma_1,\Gamma_2\subset G$ two uniform lattices without torsion.
Throughout this section we will keep the two lattices fixed and suppress them from the notation. In particular, we will denote
$\Delta\m_\sigma=m_{\Gamma_1,\sigma}-\m_{\Gamma_2,\sigma},\quad \Delta L_{\sigma}=L_{\Gamma_1,\sigma}-L_{\Gamma_2,\sigma}$,
$\Delta V=\vol(\Gamma_1\bs G)-\vol(\Gamma_2\bs G)$, $\pls=\pls_{\Gamma_1}\cup\pls_{\Gamma_2}$ and $S_\sigma=S_{\Gamma_1,\sigma}\cup S_{\Gamma_2,\sigma}$.

\subsection{Density results for a fixed $\sigma$}
As a first step, for each fixed $\sigma\in \hat{M}$ we will use the trace formula to relate the $\sigma$-representation spectrum to the $\sigma$-length spectrum. In particular, we show that we can recover the $\sigma$-length spectrum from the $\sigma$-representation spectrum and vise versa, and moreover, to do that all we need is to know one of them up to an error of density zero.

\begin{prop}\label{p:spec2length}
For any fixed $\sigma\in \hat{M}$ if
$$\lim_{T\to\infty}\frac{1}{T}\mathop{\sum_{\nu\in i\bbR}}_{|\nu|<T} |\Delta \m_\sigma(\nu)|=0,$$
then $\vol(\Gamma_1\bs G)=\vol(\Gamma_2\bs G)$ and $\Delta L_{\sigma}(\ell)=0$ for all $\ell$.
\end{prop}
\begin{proof}
Assume first that $\sigma$ is ramified.
The equality of the volumes follows from Weyl's law for the principal series (see \cite[Corollary 1]{MiatelloVargas83}).
It remains to show the equality for the $\sigma$-length spectrum. To do this we will use the trace formula with an appropriate test function.

Fix $g\in C^\infty_c(\bbR)$ even and supported on $[-1,1]$ with $g(0)=1$. Fix $\ell_0\in (0,\infty)$ and let $g_{T}(\ell)=g(T(\ell-\ell_0))+g(T(\ell+\ell_0))$ for $T$ a large parameter.
Taking the difference of the trace formulas for $\Gamma_1$ and $\Gamma_2$ applied to $g_{T}$ we get
\begin{equation}\label{e:tracediff1}
\sum_{\nu_k\in i(0,\rho)\cup \bbR} \Delta\m_{\sigma}(i\nu_k)\hat{g}_{T}(\nu_k)=\sum_{\ell\in\pls} g_{T}(\ell)\Delta L_{\sigma}(\ell).
\end{equation}
Note that the term involving the volume cancels out.

Since the function $g_{T}$ is supported on $[\ell_0-\tfrac{1}{T},\ell_0+\tfrac{1}{T}]\cup [-\ell_0-\tfrac{1}{T},-\ell_0+\tfrac{1}{T}]$, for $T$ sufficiently large the right hand side of \eqref{e:tracediff1} is given by  $\Delta L_{\sigma}(\ell_0)(1+g(2T\ell_0))$ which converges to $\Delta L_{\sigma}(\ell_0)$ as $T\to\infty$.

On the other hand, $\hat{g}_{T}(\nu)=\frac{1}{T}\hat{g}(\frac{\nu}{T})2\cos(\nu \ell_0)$ is bounded by $\frac{2}{T}|\hat{g}(\frac{\nu}{T})|$ for $\nu\in \bbR$ and by $\frac{2}{T}\hat{g}(\frac{\nu}{T})\cosh(\rho \ell_0)$ for $\nu\in i(0,\rho)$.
We can thus bound,
\begin{eqnarray*}
|\sum_{\nu_k\in i(0,\rho)\cup \bbR} \Delta\m_{\sigma}(i\nu_k)\hat{g}_{T}(\nu_k)|&\lesssim &  \frac{1}{T}\sum_{\nu_k\in i(0,\rho)} |\Delta\m_{\sigma}(i\nu_k)|\\
&&+
\frac{1}{T}\sum_{\nu_k\in \bbR} |\Delta\m_{\sigma}(i\nu_k)||\hat{g}(\frac{\nu_k}{T})|.
\end{eqnarray*}
The first (finite) sum goes to zero in the limit
and for the second sum, using the fast decay of $\hat{g}(\nu)\lesssim \frac{1}{1+|\nu|^3}$ for $\nu\in \bbR $ we can bound
\begin{eqnarray*}
\lefteqn{\frac{1}{T}\sum_{\nu_k\in \bbR} |\Delta\m_{\sigma}(i\nu_k)||\hat{g}(\frac{\nu_k}{T})|=}\\
&&= \frac{1}{T}\sum_{j=1}^\infty \sum_{|\nu_k|\in [T(j-1),Tj]} |\Delta\m_{\sigma}(i\nu_k)||\hat{g}(\frac{\nu_k}{T})|\\
&& \lesssim \sum_{j=1}^\infty \frac{1}{j^2} \left(\frac{1}{jT}\sum_{|\nu_k|\leq jT} |\Delta\m_{\sigma}(i\nu_k)|\right).
\end{eqnarray*}
From the assumption that $ \frac{1}{T}\sum_{|\nu_k|\leq T} |\Delta\m_{\sigma}(i\nu_k)|\to 0$ as $T\to\infty$ it is not hard to see that the above sum also goes to zero in the limit.
Comparing this with the right hand side, we get that $|\Delta L_\sigma(\ell_0)|=0$.

When $\sigma$ is unramified, we recall that $\m_{\sigma}(\nu)=\m(\pi_{\sigma,\nu})=\m(\pi_{w\sigma,-\nu})$ for all but finitely many values of $\nu$.
Consequently, we have that also
$$\lim_{T\to\infty}\frac{1}{T}\mathop{\sum_{\nu\in i\bbR}}_{|\nu|<T} |\Delta \m^\pm_\sigma(\nu)|=0.$$
Using the same argument we get that $\Delta L_{\sigma}+\Delta L_{w\sigma}=0$, while a similar argument with an odd test function gives $\Delta L_{\sigma}-\Delta L_{w\sigma}=0$, implying that $\Delta L_{\sigma}=\Delta L_{w\sigma}=0$.
\end{proof}
\begin{rem}
We note that the last argument using the trace formula with an odd test function is only needed when $\calH$ is real hyperbolic space with $\dim(\calH)\equiv 1\pmod{4}$. Indeed, when $\dim(\calH)\equiv 3\pmod{4}$ any ramified $\sigma\in \hat{M}$ satisfies $\chi_{\sigma}(m)=\chi_{w\sigma}(m^{-1})$. Consequently, since $\ell_{\gamma}=\ell_{\gamma^{-1}}$ and $m_{\gamma^{-1}}=m_{\gamma}^{-1}$, we get that in this case $L_{\Gamma,\sigma}(\ell)=L_{\Gamma,w\sigma}(\ell)$ automatically. We note however that this is not the case when $\dim\calH\equiv 1\pmod{4}$. Here,  $\chi_{\sigma}(m)\neq \chi_{w\sigma}(m^{-1})$ and there is no obvious reason to suspect that $L_{\Gamma,\sigma}= L_{\Gamma,w\sigma}$ in this case.
\end{rem}

\begin{prop}\label{p:length2spec}
For any ramified $\sigma\in \hat{M}$, if
\[\lim_{T\to\infty}\frac{1}{T}\mathop{\sum_{\ell\in \pls}}_{\ell\leq T} |\Delta L_{\sigma}(\ell)|=0,\]
then $\Delta\m_{\sigma}(\nu)=0$ for all $\nu\in (0,\rho)\cup i\bbR$.
When $\sigma$ is unramified we have the same result with $\Delta \m_\sigma^+(\nu)$ and $\Delta L_\sigma^+$ instead.
\end{prop}
\begin{proof}
We will write down the proof for ramified $\sigma$, the proof in the unramified case is identical.

We first show that $\Delta\m_{\sigma}(\nu)=0$ for all $\nu\in (0,\rho)$. Indeed, if not let $\nu_0\in (0,\rho)$ denote the largest element for which  $\Delta\m_{\sigma}(\nu_0)\neq 0$ and consider the test function
$$\hat{g}_T(\nu)=\frac{\nu_0\sin(T\nu)\hat{g}(\nu)}{\nu \sinh(T\nu_0)\hat{g}(i\nu_0)},$$
with $g\in C^\infty_c(\bbR)$ even, supported on $[-1,1]$ with $g(0)=1$ and satisfies that $\hat{g}(i\nu_0)\neq 0$.
Taking the difference of the trace formulas for $\Gamma_1$ and $\Gamma_2$ we get
 \begin{equation*}
\sum_{\nu_k\in i(0,\rho)\cup \bbR} \Delta\m_{\sigma}(i\nu_k)\hat{g}_T(\nu_k)=\Delta V\int_{\bbR}\hat{g}_T(\nu)\mu_{\sigma}(\nu)d\nu+\sum_{\ell\in L_\Gamma} g_T(\ell)\Delta L_{\sigma}(\ell)
\end{equation*}
For any $\nu\in(0,\rho)$ with $\nu<\nu_0$ we have that
$\hat{g}_T(i\nu)\lesssim e^{T(\nu_0-\nu)}$ and $\hat{g}_T(i\nu_0)=1$.
For $\nu\in (0,1)$ we can bound $|\hat{g}_T(\nu)|\lesssim \frac{T}{\sinh(T\nu_0)}$ and for $\nu\in [1,\infty)$ we have
$|\hat{g}_T(\nu)|\lesssim \frac{g(\nu)}{\sinh(T\nu_0)}$. Consequently we get that as $T\to\infty$ the left hand side of the trace formula converges to $\Delta\m_{\sigma}(\nu_0)$.

To estimate the right hand side, we can bound the integral
\begin{eqnarray*}|\int_{\bbR}\hat{g}_T(\nu)\mu_{\sigma}(\nu)d\nu|\lesssim \frac{T}{\sinh(\nu_0 T)}\int_0^1|\hat{g}(\nu)|\mu_{\sigma}(\nu)d\nu\\
+\frac{1}{\sinh(\nu_0 T)}\int_1^\infty |\hat{g}(\nu)|\mu_{\sigma}(\nu)d\nu\lesssim  \frac{T+1}{\sinh(\nu_0 T)}.
\end{eqnarray*}
For the sum, note that $g_T$ is given by a convolution
$$g_T(\ell)=\frac{\nu_0}{\sinh(T\nu_0)\hat{g}(i\nu_0)}g*\id_{T},$$
where $\id_{T}$ is the indicator function of $[-T,T]$. In particular, $g_{T}$ is supported on $[-T-1,T+1]$ and satisfies $|g_T(\ell)|\lesssim \frac{1}{\sinh(\nu_0 T)}$ implying that
\begin{eqnarray*}
|\sum_{\ell_j}g_T(\ell_j)\Delta L_{\sigma}(\ell_j)|\lesssim \frac{1}{\sinh(\nu_0 T)}\sum_{\ell_j\leq T+1}|\Delta L_{\sigma}(\ell_j)|\lesssim \frac{T}{\sinh(\nu_0 T)}.
\end{eqnarray*}
So as $T\to \infty$ the right hand side goes to zero implying that $\Delta\m_{\sigma}(\nu_0)=0$ as well.

Next we show that $\Delta\m_{\sigma}(i\nu)=0$ for all $\nu\in \bbR$. For this we fix $\nu_0\in [0,\infty)$ and consider the function
$g_{T}(\ell)=\frac{1}{T}g(\frac{\ell}{T})2\cos(\ell \nu_0)$, so that $\hat{g}_{T}(\nu)=\hat{g}(T(\nu-\nu_0))+\hat{g}(T(\nu+\nu_0))$. Since we already showed that $\Delta\m_{\sigma}(\nu)=0$ for $\nu\in (0,\rho)$ the difference of the trace formulas takes the form
\begin{equation*}
\sum_{\nu_k\in \bbR} \Delta\m_{\sigma}(i\nu_k)\hat{g}_{T}(\nu_k)=\Delta V\int_{\bbR}\hat{g}_{T}(\nu)\mu_{\sigma}(\nu)d\nu+\sum_{\ell_j} g_{T,x}(\ell_j)\Delta L_{\sigma}(\ell_j)
\end{equation*}
A simple change of variables shows that the integral $\int_{\bbR}\hat{g}_{T}(\nu)\mu_{\sigma}(\nu)d\nu$ goes to zero as $T\to\infty$. The sum
\begin{eqnarray}|\sum_{\ell_j} g_{T}(\ell_j)\Delta L_{\sigma}(\ell_j)|\lesssim  \frac{1}{T}\sum_{\ell_j\leq T}|\Delta L_{\sigma}(\ell_j)|
\end{eqnarray}
also goes to zero by our assumption.

Next we estimate the left hand side.
Let $\delta>0$ be such that  $\m_{\Gamma_1,\sigma}(i\nu_k)=\m_{\Gamma_2,\sigma}(i\nu_k)=0$ for all $\nu_k\in[\nu_0-\delta,\nu_0+\delta]\setminus\{\nu_0\}$. The left hand side of the trace formula can be written as
\begin{eqnarray*}
\Delta\m_{\sigma}(i\nu_{0})\hat{g}_{T}(\nu_0)+ \sum_{\nu_k\in \bbR\setminus\{\nu_0\}} \Delta\m_{\sigma}(i\nu_k)\hat{g}_{T}(\nu_k)
\end{eqnarray*}
The first term is $\Delta\m_{\sigma}(i\nu_{0})(1+o(1))$ as $T\to\infty$. For the rest of the sum we can bound $|\Delta\m_{\sigma}|\leq \m_{\Gamma_1,\sigma}+\m_{\Gamma_2,\sigma}$ and for each of the two lattices we have
\begin{eqnarray*}
\sum_{\nu_k\in \bbR\setminus\{\nu_0\}} \m_{\Gamma,\sigma}(i\nu_k)|\hat{g}_{T}(\nu_k)|&=&
\sum_{|\nu_k-\nu_0|\in (\delta,1)} \m_{\Gamma,\sigma}(i\nu_k)|\hat{g}_{T}(\nu_k)|\\
&+&\sum_{j=1}^\infty \sum_{|\nu_k-\nu_0|\in [j,j+1)} \m_{\Gamma,\sigma}(i\nu_k)|\hat{g}_{T}(\nu_k)|\\
\end{eqnarray*}
Using the fast decay of $\hat{g}$ we can bound $|\hat{g}_{T}(\nu)|\lesssim_N \frac{1}{{Tj}^N}$ for $|\nu-\nu_0|>j$ and $|\hat{g}_{T}(\nu)|\lesssim \frac{1}{T\delta}$ for $|\nu-\nu_0|\in(\delta,1)$
implying the bound
\begin{eqnarray*}
\lefteqn{\sum_{\nu_k\in \bbR\setminus\{\nu_0\}} \m_{\Gamma,\sigma}(i\nu_k)|\hat{g}_{T}(\nu_k)|}\\
&& \lesssim_N
\frac{1}{T}\left(\sum_{|\nu_k-\nu_0|\in [\delta,1]} \m_{\Gamma,\sigma}(i\nu_k)+\sum_{j=1}^\infty \frac{1}{j^N}\sum_{|\nu_k-\nu_0|\in [j,j+1]} \m_{\Gamma,\sigma}(i\nu_k)\right).\\
\end{eqnarray*}
Since for $N$ sufficiently large (depending only on the dimension) the sum
$$\sum_{j=1}^\infty \frac{1}{j^N}\sum_{|\nu_k-\nu_0|\in [j,j+1]} \m_{\Gamma,\sigma}(i\nu_k)$$
converges we get that the left hand side of the formula converges to $\Delta\m_{\sigma}(i \nu_{0})$, and since the right hand side goes to zero we have $\Delta\m_{\sigma}(i \nu_{0})=0$.
\end{proof}

Since $\m_{\Gamma,\sigma}(\nu)=\m_\Gamma(\pi_{\sigma,\nu})$ for all but finitely many values of $\nu$ we have that
$$\mathop{\sum_{\nu\in i\bbR}}_{|\nu|\leq T}|\Delta \m_\sigma(\nu)|=D_\sigma(\Gamma_1,\Gamma_2;T)+O(1).$$
Thus, combining Propositions \ref{p:spec2length} and \ref{p:length2spec} we get
\begin{thm}\label{t:Spec2Length2Spec}
For any fixed $\sigma\in \hat{M}$ the following are equivalent
\begin{enumerate}
\item  $\lim_{T\to\infty}\frac{\D_\sigma(\Gamma_1,\Gamma_2;T)}{T}=0.$

\item $\lim_{T\to\infty}\frac{1}{T}\sum_{\ell_j\leq T}|L_{\Gamma_1,\sigma}(\ell_j)-L_{\Gamma_2,\sigma}(\ell_j)|=0.$

\item $\Gamma_1$ and $\Gamma_2$ are $\sigma$-length equivalent, $\sigma$-representation equivalent and $\vol(\Gamma_1\bs G)=\vol(\Gamma_2\bs G)$.
\end{enumerate}
\end{thm}

\subsection{Proof of Theorem  \ref{t:Spec2Spec}}
The first part is a special case of Theorem \ref{t:Spec2Length2Spec} with $\sigma=1$ (recall that $\m_{\Gamma,1}(\nu)=\m_{\Gamma}(\pi_{1,\nu})$ is the multiplicity of the eigenvalue $\rho^2-\nu^2$). We also note that the second condition in Theorem \ref{t:Spec2Length2Spec} with $\sigma=1$ gives an analogous result for the $1$-length spectrum.

For the second part, let $\frak{B}\subset\hat{M}$ be a finite set and assume that
$\frac{\D_\sigma(\Gamma_1,\Gamma_2;T)}{T}\to 0$ for all $\sigma\in \hat{M}\setminus \mathfrak{B}$.
From Theorem \ref{t:Spec2Length2Spec} we get that $\vol(\Gamma_1\bs G)=\vol(\Gamma_2\bs G)$ and $L_{\Gamma_1,\sigma}=L_{\Gamma_2,\sigma}$ for all $\sigma\in \hat{M}\setminus \mathfrak{B}$. We will show that $\Gamma_1$ and $\Gamma_2$ have the same complex length spectrum and are hence representation equivalent.

Fix an arbitrary length $\ell_0\in (0,\infty)$ and holonomy  $m_0\in M$. Since there are only finitely many geodesics of length $\ell_0$ we can find $B\subset M^\#$ a small open neighborhood of $[m_0]$ satisfying that
$\{[m_{\gamma}]| \ell_\gamma=\ell_0\}\cap B\subseteq \{[m_0]\}$. Let $F\in C^\infty(M)$ denote a smooth class function supported on $B$ satisfying that $F(m_0)=1$ and that
$$\hat{F}(\sigma)=\int_{M}F(m)\chi_\sigma(m)d\mu(m)=0,$$ for all $\sigma\in \mathfrak{B}$. Since these are finitely many conditions we can clearly find such a function. We then have an expansion
\[F=\sum_{\sigma\in \hat{M}} \hat{F}(\sigma)\overline{\chi_\sigma},\]
with
 \[\sum_{\sigma\in \hat{M}} |\hat{F}(\sigma)|<\infty,\]
absolutely converges and $\hat{F}(\sigma)=0$ for all $\sigma\in \mathfrak{B}$.

Next let $g\in C^\infty_c(\bbR)$ be even and supported on a small enough neighborhood of $\ell_0$ such that $\ls_{\Gamma_1}\cup \ls_{\Gamma_2}$ intersects its support at $\{\ell_0\}$ and satisfy $g(\ell_0)=\ell_0^{-1}D(m_0a_{\ell_0})$. We then have
\begin{eqnarray*} \mathop{\sum_{[\gamma], \ell_\gamma=\ell_0}}_{m_\gamma= m_0} \frac{1}{j(\gamma)}=
\sum_{\ell\in LS_{\Gamma}}g(\ell)\ell \mathop{\sum_{[\gamma]}}_{\ell_{\gamma}=\ell}\frac{F(m_\gamma)}{j(\gamma)D(\gamma)}. \end{eqnarray*}
Expand $F(m_\gamma)$ and change the order of summation (note that all series converges absolutely) to get
 \begin{eqnarray*}
 \mathop{\sum_{[\gamma], \ell_\gamma=\ell_0}}_{m_\gamma= m_0} \frac{1}{j(\gamma)}=
 2\sum_{\sigma\in \hat{M}} \hat{F}(\sigma)\sum_{\ell\in LS_{\Gamma}}g(\ell)L_{\Gamma,\sigma}(\ell). \end{eqnarray*}
Since the right hand side is the same for $\Gamma=\Gamma_1$ and $\Gamma=\Gamma_2$ then so is the left hand side, implying that
\[\mathop{\sum_{[\gamma]\in \Gamma_1^\#}}_{(\ell_{\gamma},m_\gamma)=(\ell_0,m_0)} \frac{1}{j_{\Gamma_1}(\gamma)}=\mathop{\sum_{[\gamma]\in \Gamma_2^\#}}_{(\ell_{\gamma},m_\gamma)=(\ell_0,m_0)} \frac{1}{j_{\Gamma_2}(\gamma)}.\]
Since this is true for any pair $(\ell_0,m_0)$ we get that the complex length spectrum is the same.

\subsection{Further refinement}\label{s:FurtherRefinement}
We note that the above proof will still work if we replace the finite set $\mathfrak{B}$ with an infinite set, as long as it is sufficiently sparse so that for any
small neighborhood $B\subset M^\#$ we can find a smooth class function $F$ supported on $B$ with $\hat{F}(\sigma)=0$ for all $\sigma\in \mathfrak{B}$.

When $M=\SO(2)$ we can identify $\hat{M}$ with $\bbZ$ where  $\chi_{\sigma_n}(\theta)=e^{in \theta}$ while for $M=\SO(3)$ or $\SU(2)$ we can identify $\hat{M}$ with $\bbN$ where $\chi_{\sigma_n}(\theta)=\frac{\sin(n\theta)}{\sin(\theta)}$. For these cases the following lemma shows that we can take the set $\mathfrak{B}$ to be any set satisfying that
\begin{equation}\label{e:density}\#\{n\in \mathfrak{B}: |n|\leq T\}\lesssim T^\alpha,\mbox{ with } \alpha<1.\end{equation}

\begin{lem}
Let $\mathfrak{B}\subset \bbZ$ satisfy \eqref{e:density}.
Then for any $\delta>0$ and $\theta_0\in[0,2\pi)$ there is $f\in C^\infty(\bbZ/2\pi\bbZ)$ supported on $(\theta_0-\delta,\theta_0+\delta)$ with $f(\theta_0)=1$ and $\hat{f}(n)=0$ for all $n\in \mathfrak{B}$.
\end{lem}
\begin{proof}
Since the Fourier transform of  $f(\theta-\theta_0)$ vanishes together with the Fourier transform of $f(\theta)$ we may assume that $\theta_0=0$. Also, since we may always renormalize, it is enough to show that there is $f\in C^\infty_c(\bbR)$ not identically zero that is supported on $(-\delta,\delta)$ with $\hat{f}(n)=0$ for $n\in \mathfrak{B}$.

Let $\beta\in(\alpha,1)$ and let $f_1\in C^\infty_c(\bbR)$ be smooth and supported on $(-\delta/2,\delta/2)$ with Fourier transform satisfying $\hat{f}_1(\xi)\lesssim e^{-C|\xi|^\beta}$ for $\xi\in \bbR$. We note that a compactly supported function with this decay exists for any $\beta\in(0,1)$ (see e.g. \cite[Lemma 6]{Aubry06}) and we can make the support as small as we want by dilation. The support of $f_1$ implies that $\hat{f}_1$ can be extended to an entire function on the complex plane of exponential growth $|\hat{f}_1(z)|\lesssim \exp(\delta |z|/2)$.

Next consider the function defined by the infinite product
$$\hat{f}_2(z)=\prod_{n\in \mathfrak{B}}(1-\frac{z}{n}).$$
The condition \eqref{e:density} implies that the product converges (uniformly on compacta) to an entire function with exponential growth $|\hat{f}_2(z)|\lesssim e^{C|z|^\alpha}$. This function clearly vanishes on $\mathfrak{B}$ (but it is not the Fourier transform of a compactly supported function).

We now define $\hat{f}(z)=\hat{f}_1(z)\hat{f}_2(z)$, then $\hat{f}$ is entire of exponential growth $|\hat{f}(z)|\lesssim e^{\delta|z|/2+C|z|^\alpha}\lesssim e^{\delta|z|}$, and it still vanishes on $\mathfrak{B}$. Moreover, on the real line it decays like
  $$|\hat{f}(\xi)|\lesssim e^{C|\xi|^\alpha-|\xi|^\beta}\lesssim_N \frac{1}{|\xi|^N}.$$
Consequently, the Paley Wiener Theorem (see \cite[
Theorem  19.3]{Rudin66}) implies that $\hat{f}$ is the Fourier transform of a smooth function $f$ that is supported on $(-\delta,\delta)$.
\end{proof}


\section{Alternating trace formula}\label{s:Alternating}
In the proof of Theorem \ref{t:Spec2Spec} we used the fact that the $\sigma$-representation spectrum appearing in the trace formula is essentially given by the multiplicities of the principal series representations. The relation between the $\sigma$-length spectrum and the length spectrum is not that straight forward (even for trivial $\sigma$). In this section we develop a new formula (which is an alternating sum of trace formulas corresponding to certain virtual representations of $M$) where the geometric side involves the length spectrum directly. Precisely we show

\begin{thm}\label{t:AlternatingTrace}
Let $G$ denote a real rank one group and $\Gamma\subset G$ a uniform lattice without torsion. Let $m=\rho_G-\alpha_0(G)\in \bbN$, then there are $m+1$ virtual representations $\eta_0,\eta_1,\ldots,\eta_m$ of $M$, with $\eta_0$ the trivial representation such that
for any even $g\in C^\infty_c(\bbR)$ we have
\begin{eqnarray*}
\lefteqn{\frac{1}{2}\sum_{\ell\in \textsc{pls}_\Gamma} \ell  \mathrm{m}^o_{\Gamma}(\ell)\sum_{j=1}^\infty \frac{ g(j \ell)}{\psi(j \ell)}}\\
&& =\sum_{q=0}^{m} (-1)^q\bigg(\sum_{\nu_k\in S_{\Gamma,\sigma}} \m_{\Gamma,\eta_q}(\nu_k)\hat{g}_{q}(i\nu_k)- \vol(\Gamma\bs G)\int_\bbR \hat{g}_{q}(\nu)\mu_{\eta_q}(\nu)d\nu\bigg),
\end{eqnarray*}
where $\hat{g}_{q}(\nu)=\hat{g}(\nu+i(m-q))+\hat{g}(\nu+i(q-m))$,
and the weight function $\psi$ is given by
\begin{equation}\label{e:weight}
\psi(\ell)=\left\lbrace\begin{array}{ll} 1 & G=\SO_0(2n+1,1)\\
2\sinh(\tfrac{\ell}{2}) & G=\SO_0(2n+2,1)\\
2\sinh(\ell) & \mbox{ otherwise}\\
\end{array}\right.
\end{equation}
\end{thm}
\begin{proof}
In order to derive this alternating formula we expand the Weyl discriminant \eqref{e:WeylDiscriminant} as a sum of characters of representations of $M$.
Specifically, we will show in Proposition \ref{p:Discriminant} below that
\begin{equation}\label{e:Discrimiant} D(\gamma)=\psi(\ell_\gamma)\sum_{q=0}^{m}(-1)^q(e^{(m-q)\ell_\gamma}+e^{(q-m)\ell_\gamma})\chi_{\eta_{q}}(m_\gamma),\end{equation}
with $\eta_0,\ldots,\eta_m$ certain virtual representations of $M$ with $\eta_0$ trivial.

Now, since any $\gamma\in \Gamma$ can be written in a unique way as $\gamma=\delta^j$ with $\delta\in \Gamma$ primitive, we have
\begin{eqnarray*}
\sum_{\ell\in \pls_\Gamma}\ell \m^o_\Gamma(\ell)\sum_{j=1}^\infty \frac{g(j \ell)}{\psi(j\ell)}&=&
\sum_{\ell\in \ls_\Gamma}\frac{g(\ell)}{\psi(\ell)}\mathop{\sum_{[\gamma]}}_{\ell_\gamma=\ell}\frac{\ell }{j(\gamma)}\\
&=& \sum_{\ell\in \ls_\Gamma }\frac{g(\ell)}{\psi(\ell)}\mathop{\sum_{[\gamma]}}_{\ell_\gamma=\ell}\frac{\ell D(\gamma)}{j(\gamma)D(\gamma)}
\end{eqnarray*}
Plugging in the expansion for $D(\gamma)=\overline{D(\gamma)}$ from \eqref{e:Discrimiant} in the numerator we get
\begin{eqnarray*}
\sum_{\ell\in \pls_\Gamma}\ell \m^o_\Gamma(\ell)\sum_{j=1}^\infty \frac{g(j \ell)}{\psi(j\ell)}
&=& \sum_{q=0}^{m}(-1)^q\sum_{\ell\in \ls_\Gamma }g_q(\ell)\mathop{\sum_{[\gamma]}}_{\ell_\gamma=\ell}\frac{\ell \overline{\chi_{\eta_q}(m_\gamma)}}{j(\gamma)D(\gamma)}\\
&=& 2\sum_{q=0}^{m}(-1)^q\sum_{\ell\in \ls_\Gamma }g_q(\ell)L_{\Gamma,\eta_q}(m_\gamma)\\
\end{eqnarray*}
where $g_q(\ell)=g(\ell)(e^{(m-q)\ell}+e^{(q-m)\ell})$.
We conclude the proof by applying the trace formula attached to each of the virtual representation $\eta_q$ separately.
\end{proof}
\begin{rem}
For the odd dimensional orthogonal groups the representations $\eta_1,\ldots,\eta_m$ are actual irreducible representation.
For the other groups, it is also possible to obtain a similar formula using irreducible representations instead of virtual representations by using appropriate linear combinations of the $\hat{g}_{q}$'s.
\end{rem}
The rest of this section will be devoted to the proof of the expansion  \eqref{e:Discrimiant} for the Weyl discriminant.

\subsection{The Adjoint representation}
In order to expand the Weyl discriminant
$$D(\gamma)=e^{\rho\ell_\gamma}|\det((\Ad(m_\gamma a_{\ell_\gamma})^{-1}-I)_{|\frakn})|,$$
as a sum of characters we first need to understand  $\Ad(MA)$ and in particular its restriction to $\frakn_1$ and $\frakn_2$. Denote by $n_1=\dim\frakn_1$ and $n_2=\dim\frakn_2$. From the definition of $\frakn_1$ and $\frakn_2$, for $a_\ell=\exp(\ell H)\in A$ we have that $\Ad(a_\ell)_{|_{\frakn_1}}=e^\ell I$ and $\Ad(a_\ell)_{|_{\frakn_2}}=e^{2\ell} I$.
The following proposition describes the action of $\Ad(M)$.
\begin{prop}\label{p:Adjoint}
For the orthogonal group $\SO_0(n+1,1),\;M\cong \SO(n)$ and the restriction of $\Ad(M)$ to $\mathfrak{n}=\frakn_1$ is the natural isomorphism.
For the unitary and symplectic groups the restriction of $\Ad(M)$ to $\mathfrak{n}_j,\; j=1,2$ gives a homomorphism  $\iota_j:M\to \SO(n_j)$.
For $G=FII,\; M=\mathrm{Spin}(7)$ and the restriction of $\Ad(M)$ to $\frakn_1$ and $\frakn_2$ is the $8$-dimensional spin representation and the $7$-dimensional orthogonal representation respectively.
\end{prop}
\begin{proof}
For the orthogonal group this is clear.
For the unitary group $G=\SU(n+1,1)$, we write its Lie algebra as
\[\mathfrak{g}=\set{\begin{pmatrix} a & b\\b^* & d\end{pmatrix}| a\in \Mat_{n+1}(\bbC),b\in \bbC^{n+1}, a^*=-a,d=-\Tr(a)},\]
and the Lie algebras in the Cartan decomposition $\frakg=\frakp\oplus\frakk$ are given by $\mathfrak{p}=\set{\begin{pmatrix} 0 & b\\b^* & 0\end{pmatrix}\in \mathfrak{g}}$,
and $\frakk=\set{\begin{pmatrix} a & 0\\0 & d\end{pmatrix}\in \mathfrak{g}}$.
We write any $b\in \bbC^{n+1}$ as $b=\begin{pmatrix} v\\ c\end{pmatrix}$ with $v\in \bbC^{n}$ and $c\in \bbC$.
Then $\mathfrak{a}\subset \frakp$ is the real subspace
$\fraka=\set{\begin{pmatrix} 0 & b\\b^* & 0\end{pmatrix}|b=\begin{pmatrix}0\\c\end{pmatrix},\; c\in \bbR}$
and the root spaces $\frakn_1$ and $\frakn_2$ are given by
\[\mathfrak{n}_1=\set{X_v=\begin{pmatrix} 0 & v & v\\ -v^*& 0 & 0\\ v^* &0 &0\end{pmatrix}\bigg| v\in \bbC^{n}},\]
of dimension $n_1=2n$ and
\[\mathfrak{n}_2=\set{\begin{pmatrix} 0 & 0 & 0\\ 0 & -d & d\\ 0 & -d & d \end{pmatrix}\bigg| d\in \bbC, d+d^*=0},\]
which is one dimensional.
We get that $M$, the centralizer of $A$ in $K$, is given by
\[M=\set{\begin{pmatrix} u & 0 & 0\\ 0 & x & 0\\ 0 &0 & x\end{pmatrix}\bigg| u\in \mathrm{U}(n),\;x\in U(1), \; x^2\det(u)=1}.\]
A simple computation then shows that the adjoint action of $M$ is trivial on $\frakn_2$ and that
$\Ad(m)X_v=m X_v m^{-1}=X_{x^*uv}$. Fixing the standard basis $X_{e_j},X_{ie_j},\;j=1,\ldots,n$ for $\frak{n}_1$ (recalling the natural inclusion $U(n)\subseteq \SO(2n)$) we get a homomorphism from $M$ to $\SO(2n)$.

For $G=\Sp(n,1)$ repeating the same arguments replacing $\bbC$ with the quaternions $\bbH$ we get
\[\mathfrak{n}_1=\set{X_v=\begin{pmatrix} 0 & v & v\\ -v^*& 0 & 0\\ v^* &0 &0\end{pmatrix}\bigg| v\in \bbH^{n-1}},\]
of dimension $n_1=4(n-1)$,
 \[\mathfrak{n}_2=\set{Y_d=\begin{pmatrix} 0 & 0 & 0\\ 0 & -d & d\\ 0 & -d & d \end{pmatrix}\bigg| d\in \bbH, d+d^*=0},\]
of dimension $n_2=3$, and
\[M=\set{\begin{pmatrix} u & 0 & 0\\ 0 & x & 0\\ 0 &0 & x\end{pmatrix}\bigg| u\in \Sp(n-1),\;x\in \Sp(1)\; \det(u)x^2=1}.\]
On $\frakn_1$ we have $\Ad(m)X_v=X_{uvx^*}$ and fixing a suitable basis for $\bbH^n\cong\bbC^{2n}\cong \bbR^{4n}$ 
gives a homomorphism from $M$ to $\SO(4(n-1))$. On $\frakn_2$ the action is given by $\Ad(m)Y_d=Y_{xdx^*}$
and choosing a suitable basis this gives a homomorphism into $\SO(3)$.

Finally, for $G=FII$ we have that $\frakn_1$ is $8$-dimensional and $\frakn_2$ is $7$-dimensional. The smallest irreducible representation of $M=\rm{Spin}(7)$ is the orthogonal representation which is $7$-dimensional and the only $8$-dimensional irreducible representation is the spin representation.
Let $\mathfrak{t}$ denote a maximal commutative subspace of $\frakm$ so that $\mathfrak{h}=\mathfrak{t}\oplus \fraka$ is a Cartan algebra for $\frakg$.
Since the root spaces for $\mathfrak{h}$ are one dimensional, any subspace of $\frakn_j$ on which $\ad(\mathfrak{t})$ acts trivially is at most one dimensional.
Consequently, the action of $\ad(\frakm)$ is not trivial on $\frakn_1$ and $\frakn_2$, and hence the restriction of $\Ad(M)$ to $\frakn_2$ gives the orthogonal representation and the restriction to $\frakn_1$ is either the spin representation or a direct sum of the orthogonal and trivial representation.
Finally, we note that the latter cannot hold as it would imply that the center of $M$ acts trivially on $\frakn$.

\end{proof}
\begin{rem}
For the unitary and symplectic groups the homomorphism $\iota_1:M\to\SO(n_1)$ is not the standard inclusion $\mathrm{U}(n)\subset \SO(2n)$ (respectively, $\Sp(n-1)\subseteq \SO(4(n-1)$). In particular, when $n=2$ this homomorphism has a nontrivial kernel. The homomorphism $\iota_2$ is trivial for the unitary group and factors through the natural isomorphism $\Sp(1)\cong \SO(3)$ for the symplectic group.
\end{rem}

\begin{cor}\label{c:Discriminant}
The Weyl discriminant can be factored as
\[D(\gamma)=|\det(e^{-\ell_\gamma/2}I_{n_1}-e^{\ell_\gamma/2}\iota_1(m_\gamma) )||\det(e^{-\ell_\gamma}I_{n_2}-e^{\ell_\gamma}\iota_2(m_\gamma) )|.\]
\end{cor}

\subsection{Representations of the orthogonal groups}
Next, we want to expand each one of these determinants as a linear combination of irreducible characters.
To do this we recall some facts about the representation theory of the orthogonal groups that we will need.
We refer to \cite[Chapter V]{Knapp02} for background and more details on representation theory of compact groups.

The irreducible representations of the orthogonal group are parameterized by their highest weights: When $n=2m$ the possible highest weighs are
\[\widehat{\mathrm{so}}(2m)=\{\sum_ja_je_j:\;a_1\geq \ldots\geq a_{n-1}\geq |a_n|, a_i-a_j\in \bbZ,\;2a_j\in \bbZ\},\] and when $n=2m+1$ they are
\[\widehat{\mathrm{so}}(2m+1)=\{\sum_ja_je_j:\;a_1\geq \ldots\geq a_{n}\geq 0, a_i-a_j\in \bbZ,\;2a_j\in \bbZ\}.\]
There is a one to one correspondence between these highest weights and the irreducible representations of the simply connected Lie group $\mathrm{Spin}(n)$; a representation of $\mathrm{Spin}(n)$ factors through a representation of $\SO(n)$ if and only if all the coefficient are integral.

For any $\theta\in (\bbZ/2\pi \bbZ)^m$ let
\[u_\theta=\begin{pmatrix} R(\theta_1) & &\\& \ddots &\\
& & R(\theta_m)\end{pmatrix}\in \SO(2m),\]
with $R(\theta_j)=\left(\begin{smallmatrix} \cos(\theta_j) &\sin(\theta_j)\\ -\sin(\theta_j)& \cos(\theta_j)\end{smallmatrix}\right)$.
When $n=2m$ is even, the map $\theta\mapsto u_\theta$ is an isomorphism of $(\bbZ/2\pi \bbZ)^m$ with the maximal torus of $\SO(2m)$. When $n=2m+1$ is odd this isomorphism is given by  $\theta\mapsto \tilde{u}_\theta= \smatrix{u_\theta}{0}{0}{1}$.

For any weight $\lambda$, let $\chi_\lambda$ denote the character of the irreducible representation of highest weight $\lambda$ (that we may think of as a function on $(\bbZ/2\pi\bbZ)^m$ by restriction to the maximal torus).
We recall the Weyl character formula:
\begin{equation}\label{e:WeylCharacter}
\chi_\lambda(\theta)=\frac{1}{D(\theta)}\sum_{w\in W}\sgn(w)\xi_{w(\lambda+\delta)}(\theta),\end{equation}
where $\xi_{\lambda}(\theta)=\exp(i\sum_j a_j\theta_j)$, $W$ is the Weyl group of $\SO(n)$, $D$ is the Weyl discriminant of $\SO(n)$ given by
\[D(\theta)=\xi_{\delta}(\theta)\prod_{\alpha\in \Delta^+}(1-\xi_{-\alpha}(\theta))=\sum_{w\in W}\sgn(w)\xi_{w\delta}(\theta),\]
and $\delta=\tfrac{1}{2}\sum_{\alpha\in \Delta^+}\alpha$ is half the sum of the positive roots.


For any fixed $\ell\in \bbR$ consider the function
\[F_{n}(\ell,\cdot):\SO(n)\to\bbR,\quad F(\ell,u)=|\det(e^{-\ell/2}I_{n}-e^{\ell/2}u)|.\]
This is clearly a class function on $\SO(n)$ and we can write it as a linear combination of irreducible characters.
\begin{prop}\label{p:Deven}
When $n=2m+1$ is odd
\begin{equation}\label{e:FluOdd}
F_{n}(\ell,u)= 2\sinh(\tfrac{\ell}{2})\sum_{q=0}^{m}(-1)^q\left( \sum_{k=q-m}^{m-q}e^{k\ell}\right)\chi_{\tau_q}(u)
\end{equation}
where $\tau_q$ denotes the irreducible representation of $\SO(2m+1)$ of highest weight $e_1+\ldots+e_q$

When $n=2m$ is even
\begin{eqnarray}\label{e:FluEven}
F_n(\ell,u)&=& \sum_{q=0}^{m} (-1)^q(e^{(m-q)\ell}+e^{(q-m)\ell})  \chi_{\sigma_{q}}(u)\\
\nonumber &=&\sum_{q=0}^{m}(-1)^q\left( \sum_{k=q-m}^{m-q}e^{k\ell}\right)\chi_{\tau_q}(u)
\end{eqnarray}
where $\sigma_q,\;0\leq q<m$ denotes the irreducible representation of $\SO(2m)$ of highest weight $e_1+\ldots+e_q$, $\sigma_m=\sigma_m^+\oplus \sigma_m^-$ with $\sigma_m^\pm$ of highest weight $e_1+\ldots+ e_{m-1} \pm e_m$. In the second line, $\chi_{\tau}$ is evaluated at $u\in \SO(2m)$ via the natural inclusion of $\SO(2m)\subseteq\SO(2m+1)$.

\end{prop}
\begin{proof}
We first treat the even case.
Evaluating $F_{2m}(\ell,\cdot)$ at $u_\theta$ we get
\begin{eqnarray*}
F_{2m}(\ell, u_\theta)&=& |\det(e^{-\ell/2}I_{2m}+e^{\ell/2}u_\theta)|\\
&=& \prod_{j=1}^m\big(2\cosh(\ell)-2\cos(\theta_j)\big)\\
&=& \sum_{k=0}^m(-1)^k (2\cosh(\ell))^{m-k}S_{m,k}(\theta)
\end{eqnarray*}
where
\begin{equation}\label{e:Smk}
S_{m,k}(\theta)=\sum_{1\leq j_1\leq\ldots\leq j_k\leq m}\left(\prod_{i=1}^k 2\cos(\theta_{j_i})\right).
\end{equation}

We can think of the functions $S_{m,k}$ as class functions on $\SO(2m)$ and write them as a linear combination of irreducible characters. To do this, for any $n,k\in \bbN$ with $2k\leq n$ we define
\begin{equation}\label{e:N0} N_0(n,k)=\#\{(j_1,\ldots,j_k)\in \{1,\ldots,n-1\}^k|j_{i+1}\geq j_i+2\}\end{equation}
\begin{equation}\label{e:N1}
N(n,k)=\left\lbrace\begin{array}{cc} 1 & k=0\\ N_0(n,k)+N_0(n-2,k-1) & k\geq 1\end{array}\right.
\end{equation}
We show in Lemma \ref{l:Smk2} below that
\begin{equation} \label{e:Smk2}
S_{m,k}=\sum_{j=0}^{[k/2]}(-1)^j N(m+2j-k,j)\chi_{\sigma_{k-2j}}.
\end{equation}
Plugging this in the above expression we get
\begin{eqnarray*}
\lefteqn{F_{2m}(\ell,u_\theta)=\sum_{k=0}^m(-1)^k (2\cosh(\ell))^{m-k}S_{m,k}}\\
&&=\sum_{k=0}^m\sum_{j=0}^{[k/2]} (-1)^{k+j} (2\cosh(\ell))^{m-k} N(m+2j-k,j)\chi_{\sigma_{k-2j}}\\
&&=(-1)^m\sum_{k=0}^{m}\left(\sum_{j=0}^{[k/2]}(-1)^{j-k} (2\cosh(\ell))^{k-2j}N(k,j)\right)\chi_{\sigma_{m-k}}
\end{eqnarray*}
The result now follows from the identity
\[\sum_{j=0}^{[k/2]}(-1)^{j} (2\cosh(\ell))^{k-2j}N(k,j)=2\cosh(k\ell),\]
which can be easily proved by induction from the recursion relation
\begin{equation}\label{e:Nrecursion}
N(n,k)=N(n-1,k)+N(n-2,k-1),\quad \forall\; 1\leq k\leq \frac{n-1}{2},
\end{equation}
together with the simple observation that $N(2k,k)=2$ for all $k\geq 1$.

The second line of \eqref{e:FluEven} is a consequence of the relation
\begin{equation}\label{e:KMrelation}
\chi_{\tau_k}=\left\lbrace\begin{array}{ll} \chi_{\sigma_0} & k=0\\
\chi_{\sigma_k}+\chi_{\sigma_{k-1}}
 & 1\leq k\leq m\\
\end{array}\right..
\end{equation}
which follows from the decomposition of the restriction of $\tau_k$ to $\SO(2m)$ into irreducible representations of $\SO(2m)$.

Now for the odd case, evaluating $F_{2m+1}(\ell,\cdot)$ on $\tilde{u}_\theta=\begin{pmatrix} u_\theta &0\\ 0& 1\end{pmatrix}$ gives
\begin{eqnarray*}
F_{2m+1}(\ell, \tilde{u}_\theta)&=& |\det(e^{-\ell/2}I_{2m+1}-e^{\ell/2}\tilde{u}_\theta)|\\
&=& 2\sinh(\ell/2)\prod_{j=1}^m\big(2\cosh(\ell)-2\cos(\theta_j)\big)\\
&=& 2\sinh(\ell/2)F_{2m}(\ell, u_\theta),
\end{eqnarray*}
and the result follows directly from the even case.
We note that using \eqref{e:KMrelation} we can also write
\begin{eqnarray*}
F_{2m+1}(\ell, \tilde{u}_\theta)
&=& 2\sinh(\ell/2)\sum_{q=0}^{m} (-1)^q(e^{(m-q)\ell}+e^{(q-m)\ell})  \chi_{\eta_{q}}(u)
\end{eqnarray*}
with $\eta_q$ the virtual representation of $\SO(2m+1)$ whose restriction to $\SO(2m)$ is $\sigma_q$.
\end{proof}

We still need to prove the identity \eqref{e:Smk2} and the recursion \eqref{e:Nrecursion}.
\begin{lem}
The combinatorial terms $N(n,k)$ defined in \eqref{e:N1} satisfy
\[N(n,k)=N(n-1,k)+N(n-2,k-1),\]
for all $1\leq k\leq \frac{n-1}{2}$.
\end{lem}
\begin{proof}
In the definition of $N_0(n,k)$, for any choice of $j_1\in \{1,\ldots, n-2k+1\}$ there are $N_0(n-j-1,k-1)$ choices for $j_2,\ldots,j_k$, leading to the recursion relation
\[N_0(n,k)=\sum_{j=1}^{n-2k+1}N_0(n-j-1,k-1).\]
Now, for $k=1$, $N(n,1)=n=n-1+1=N(n-1,1)+N(n-2,0)$ is trivial and for $k\geq 2$  substituting $N(n,k)=N_0(n,k)+N_0(n-1,k-2)$ we get
\[N(n,k)= \sum_{j=1}^{n-2k+1}N(n-j-1,k-1)=\sum_{j=2(k-1)}^{n-2}N(j,k-1),\]
implying that $N(n,k)-N(n-1,k)=N(n-2,k-1)$.
\end{proof}

\begin{lem}\label{l:Smk2}
The functions $S_{m,k}$ defined in \eqref{e:Smk} satisfy \[S_{m,k}=\sum_{j=0}^{[k/2]}(-1)^j N(m+2j-k,j)\chi_{\sigma_{k-2j}},\]
 (as class functions on $\SO(2m)$).
\end{lem}
\begin{proof}
To simplify notation we denote by $\chi_k=\chi_{\sigma_k}$ and $\chi_m^\pm=\chi_{\sigma_m^\pm}$.
Let $W$ denote the Weyl group of $\SO(2m)$, which acts on the weights $e_1,\ldots, e_m$ by permutations and even sign changes.
We recall the Weyl character formula
\[\chi_k(u)=\frac{1}{D(u)}\sum_{w\in W}\sgn(w)\xi_{w(\delta+e_1+\ldots+e_k)}(u),\]
where $\xi_{e_j}(u_\theta)=e^{i\theta_j}$,
$\delta=\sum_{j=1}^m(m-j)e_j$ is half the sum of the positive roots, and $D(u)=\sum_{w\in W}\sgn(w)\xi_{w\delta}(u)$ denotes the Weyl discriminant of $\SO(2m)$.

We first prove the formula for $k<m$. Let
$$W^+=\{w\in W|\sgn(w)=1\}$$
denote the subgroup of positive elements and observe that we can write
$$S_{m,k}=\frac{1}{N_k}\sum_{s\in W^+}\xi_{s(e_1+\ldots+e_k)},$$
where
\[N_k=\#\{w\in W^+|w(e_1+\ldots+e_k)=e_1+\ldots+e_k\}.\]
Multiplying by the Weyl discriminant we get
\begin{eqnarray*}
D S_{m,k}&=&\frac{1}{N_k}\sum_{w\in W}\sgn(w)\xi_{w\delta}\sum_{s\in W^+}\xi_{s(e_1+\ldots+e_k)}\\
&=& \sum_{\lambda\in E_k}\sum_{w\in W}\sgn(w)\xi_{w(\delta+\lambda)}
\end{eqnarray*}
where $E_k=\{\pm e_{j_1}\pm\ldots \pm e_{j_k}|1\leq j_1<\ldots<j_k\leq m\}$. To get this last equality, for each $\lambda\in E_k$ and $w\in W$ we collected together the $N_k$ elements $s\in W^+$ satisfying that $s(e_1+\ldots+e_k)=w\lambda$.

We can write any $\lambda\in E_k$ as $\lambda=\sum_{j=1}^m \mu_j e_j$ with $\mu_j\in\{-1,0,1\}$ (with $k$ nonzero entries). We note that if $\mu_{j+1}=\mu_j+1$ for some $0\leq j\leq m-1$ or $\mu_{j+2}=\mu_j+2$ for some $0\leq j\leq m-2$ then the transposition $w_{j,j+1}$ (respectively $w_{j,j+2}$) of $e_j$ and $e_{j+1}$ (respectively $e_{j+2}$) sends $\delta+\lambda$ to itself. This implies that for any such weight $\lambda$ the sum $\sum_{w\in W}\sgn(w)\xi_{w(\delta+\lambda)}=0$.
Consequently, the only weights $\lambda\in E_k$ that contribute to the sum are of the form
\begin{equation}\label{e:weights1}
\lambda=e_1+\ldots+e_{k-2\ell}-e_{j_1}+e_{j_1+1}+\ldots-e_{j_\ell}+e_{j_\ell+1},\end{equation}
for some $0\leq \ell \leq \tfrac{k}{2}$ where $j_1,\ldots,j_{\ell}$ satisfy
\begin{equation*}\label{e:goodindices2}
k-2\ell+1\leq j_1<j_1+1<j_2<\ldots< j_{\ell}\leq  m-1
\end{equation*}
or of the form
\begin{equation}\label{e:weights2}
\lambda=e_1+\ldots+e_{k-2\ell}-e_{j_1}+e_{j_1+1}+\ldots-e_{j_\ell}-e_{j_{\ell}+1},\end{equation}
for some $0\leq \ell \leq \tfrac{k}{2}$ where $j_1,\ldots,j_{\ell}$ satisfy
\begin{equation*}\label{e:goodindices3}
k-2\ell+1\leq j_1<j_1+1<j_2<\ldots< j_{\ell}=m-1.
\end{equation*}
For each $\lambda$ satisfying \eqref{e:weights1} or \eqref{e:weights2} let $w_\lambda$ denote the composition of the transpositions $w_{j_i,j_i+1},\;i=1,\ldots,\ell$. We then have that $\sgn(w_\lambda)=(-1)^\ell$ and
$$w_\lambda(\delta+\lambda)=\delta+e_1+\ldots+e_{k-2\ell},$$
implying that
$\frac{1}{D}\sum_{w\in W}\xi_{w(\rho+\lambda)}=(-1)^\ell\chi_{k-2\ell}$.

Now, for each $1\leq \ell\leq \frac{k}{2}$ there are precisely $N_0(m-k+2\ell,\ell)$ weights $\lambda\in E_k$ satisfying \eqref{e:weights1} and $N_0(m-k+2\ell-2,\ell-1)$ weights satisfying \eqref{e:weights2}. We can thus conclude that for any $k<m$
\begin{eqnarray*}
S_{m,k}
=\sum_{\ell=0}^{[k/2]}(-1)^\ell N(m+2\ell-k,\ell)\chi_{k-2\ell}
\end{eqnarray*}

For $k=m$ we need to make small adjustments to the argument as there is no $w\in W$ such that
$s(e_1+\ldots+e_m)=e_1+\ldots+e_{m-1}-e_m$. Let $W_0$ denote the group of all sign changes (this is not a subgroup of the Weyl group as we allow odd sign changes as well). We can write
\[S_{m,m}=\sum_{s\in W_0}\xi_{s(e_1+\ldots+e_m)},\]
and as before
\begin{eqnarray*}
D S_{m,m}&=&\sum_{w\in W}\sgn(w)\xi_{w\delta}\sum_{s\in W_0}\xi_{s(e_1+\ldots+e_m)}\\
&=& \sum_{\lambda\in E_m}\sum_{w\in W}\sgn(w)\xi_{w(\delta+\lambda)}
\end{eqnarray*}
where $E_m=\{\pm e_1\pm\ldots\pm e_m\}$ as before.

In this case, we get a contribution from all $\lambda\in E_m$ satisfying \eqref{e:weights1} and \eqref{e:weights2}, but also from the weight $\lambda=e_1+\ldots+e_{m-1}-e_m$, giving the formula
\begin{eqnarray*}S_{m,m}&=&\chi_m^-+\chi_m^++\sum_{\ell= 1}^{[m/2]}(-1)^\ell N(2\ell,\ell)\chi_{m-2\ell}\\
&=& \sum_{\ell= 0}^{[m/2]}(-1)^\ell N(2\ell,\ell)\chi_{m-2\ell}\end{eqnarray*}
\end{proof}
\subsection{The Weyl discriminant}
Combining these results we get the following expression for the Weyl discriminant.
\begin{prop}\label{p:Discriminant}
The Weyl discriminant can be written as
\begin{equation} D(\gamma)=\psi(\ell_\gamma)\sum_{q=0}^{m}(-1)^q(e^{(m-q)\ell_\gamma}+e^{(q-m)\ell_\gamma})\chi_{\eta_{q}}(m_\gamma),\end{equation}
with $m=\rho-\alpha_0(G)$, $\eta_0,\ldots,\eta_m$ virtual representations of $M$ with $\eta_0$ trivial, and $\psi(\ell)$ the weight function in \eqref{e:weight}.
\end{prop}
\begin{proof}
We prove it separately for the orthogonal, unitary, symplectic and exceptional groups.

\emph{Orthogonal groups:}
For $G=\SO_0(n+1,1)$ we have that $\rho=\tfrac{n}{2}$ and $m=[\tfrac{n}{2}]$. The formula follows immediately from Corollary \ref{c:Discriminant} and Proposition \ref{p:Deven} where $\eta_q=\sigma_q$ when $n=2m$, and it is the virtual representation of $\SO(2m+1)$ whose restriction to $\SO(2m)$ is $\sigma_q$ when $n=2m+1$.

\emph{Unitary groups:} For  $G=\SU(n+1,1)$ we have $\rho=n+1$ and $m=n$. Here, Corollary \ref{c:Discriminant} implies that
\[D(\gamma)=F_{2n}(\ell_\gamma,\iota_1(m_\gamma))|2\sinh(\ell_\gamma)|,\]
with $\iota_1:M\to\SO(2n)$ as in Proposition \ref{p:Adjoint}.
The result follows from Proposition \ref{p:Deven} with $\eta_q=\sigma_q\circ \iota_1$.

\emph{Symplectic groups:} For $G=\Sp(n,1)$ we have that $\rho=2n+1,\;m=2n$ and the formula in Corollary \ref{c:Discriminant} is
\[D(\gamma)=F_{4(n-1)}(\ell_\gamma,\iota_1(m_\gamma))F_3(2\ell_\gamma,\iota_2(m_\gamma)),\]
where $\iota_1:M\to \SO(4n-4)$ and $\iota_2:M\to \SO(3)$ are as in Proposition \ref{p:Adjoint}.
Let $\sigma_0,\ldots,\sigma_{2n-2}$ and $\tau_0,\tau_1$ denote the representations of $\SO(4n-4)$ and $\SO(3)$ appearing in \eqref{e:FluEven} and \eqref{e:FluOdd}. Consider the representations of $M=\Sp(n-1)$ given by $\tilde{\sigma}_q=\sigma_q\circ\iota_1$ and $\tilde\tau_q=\tau_q\circ\iota_2$.
Using the expansions \eqref{e:FluOdd} and \eqref{e:FluEven} for $F_3(2\ell,\cdot)$ and $F_{4(n-1)}(\ell,\cdot)$ we get
\begin{eqnarray*}
D(\gamma)&=&2\sinh(\ell_\gamma)\bigg(\sum_{q=0}^{m-2}(-1)^q(e^{(m-q)\ell_\gamma}+e^{(q-m)\ell_\gamma})\chi_{\tilde{\sigma}_{q}}(m_\gamma)\\
&+&\sum_{q=0}^{m-2}(-1)^q(e^{(m-4-q)\ell_\gamma}+e^{(q-m+4)\ell_\gamma})\chi_{\tilde{\sigma}_{q}}(m_\gamma)\\
&+&\sum_{q=0}^{m-2}(-1)^q(e^{(m-2-q)\ell_\gamma}+e^{(q-m+2)\ell_\gamma})\chi_{\tilde{\sigma}_{q}}(m_\gamma)\\
&-&\sum_{q=0}^{m-2}(-1)^q(e^{(m-2-q)\ell_\gamma}+e^{(q-m-2)\ell_\gamma})\chi_{\tilde\sigma_q\otimes \tilde\tau_1}(m_\gamma)\bigg)\\
&=&2\sinh(\ell_\gamma)\sum_{q=0}^{m}(-1)^q(e^{(m-q)\ell_\gamma}+e^{(q-m)\ell_\gamma})\chi_{\eta_{q}}(m_\gamma)
\end{eqnarray*}
with $\eta_0$ the trivial representation and $\eta_q$ the virtual representation
\[\eta_q=\left\lbrace\begin{array}{ll}
\tilde\sigma_1 & q=1\\
\tilde\sigma_q+\tilde\sigma_{q-2}-\tilde{\sigma}_{q-2}\otimes\tilde\tau_1& 2\leq q<4\\
\tilde\sigma_q+\tilde\sigma_{q-2}+\tilde\sigma_{q-4}-\tilde{\sigma}_{q-2}\otimes\tilde\tau_1& 4\leq q\leq m-2\\
2\tilde\sigma_{m-3}+\tilde\sigma_{m-5}-\tilde{\sigma}_{m-3}\otimes\tilde\tau_1&  q=m-1\\
\tilde\sigma_{m-4}+\tilde\sigma_{m-2}-\tilde{\sigma}_{m-2}\otimes\tilde\tau_1&  q=m\\
\end{array}\right.\]

\emph{Exceptional group:} For $G=FII$ we have $\rho=11,\;m=10$ and we can write
\[D(\gamma)=F_{8}(\ell_\gamma,\iota_1(m_\gamma))F_7(2\ell_\gamma,\iota_2(m_\gamma)),\]
where $\iota_1:M\to \GL_8(\bbC)$ and $\iota_2:M\to \SO(7)$ are the spin and orthogonal representations of $M=\mathrm{Spin}(7)$.
Note that the function $F_8(\ell,\cdot)$ is actually a class function on $GL_8(\bbC)$ so this makes sense.
However, $\iota_1(m_\gamma)$ is not in the orthogonal group so we can't use Proposition \ref{p:Deven} directly.
Nevertheless, since the weights of the spin representations are $\{s\lambda|\;s\in W_0=\{\pm 1\}^3\}$ with $\lambda=\tfrac{1}{2}(e_1+e_2 +e_3)$ we have
\begin{eqnarray*}
F_{8}(\ell,\iota_1(\cdot))&=& \prod_{s\in W_0}(e^{-\ell/2}-e^{\ell/2}\xi_{s\lambda})\\
&=& 2\cosh(4\ell)-2\cosh(3\ell)\!\!\sum_{s\in W_0}\!\!\xi_{s\lambda}+2\cosh(2\ell)\!\!\!\!\!\!\sum_{\{s_1,s_2\}\subseteq W_0}\!\!\!\!\!\!\xi_{(s_1+s_2)\lambda}\\
&-&2\cosh(\ell)\!\!\!\!\!\sum_{\{s_1,s_2,s_3\}\subseteq W_0}\xi_{(s_1+s_2+s_3)\lambda}+\!\!\!\sum_{\{s_1,\ldots,s_4\}\subseteq W_0}\!\!\!\xi_{(s_1+s_2+s_3+s_4)\lambda}
\end{eqnarray*}
Using similar augments to the ones used in the proof of Lemma \ref{l:Smk2} we identify each of the above sums as a linear combination of irreducible characters of $\mathrm{Spin}(7)$ to get
\begin{eqnarray*}
\lefteqn{F_{8}(\ell,\iota_1(\cdot))= 2\cosh(4\ell)-2\cosh(3\ell)\chi_\lambda+2\cosh(2\ell)(\chi_{e_1+e_2}+\chi_{e_1}+1)}\\
&&-2\cosh(\ell)(\chi_\lambda+\chi_{\lambda+e_1})+2(\chi_{2e_1}+\chi_{e_1+e_2+e_3}+\chi_{e_1+e_2}+\chi_{e_1}+1)
\end{eqnarray*}
For $F_7(2\ell_\gamma,\iota_2(m_\gamma))$  we can use Proposition  \ref{p:Deven} directly to get
\[F_7(2\ell_\gamma,\iota_2(m_\gamma))=2\sinh(\ell)\sum_{q=0}^3\big(\sum_{k=q-3}^{3-q}e^{2k\ell_\gamma}\big)\chi_{\tau_k}(m_\gamma).\]
Multiplying the two expansions and collecting together terms with the same coefficients as we did for the symplectic groups we get that indeed
\[D(\gamma)=2\sinh(\ell_\gamma)\sum_{q=0}^{10}(e^{(10-q)\ell_\gamma}+e^{(q-10)\ell_\gamma})\chi_{\eta_q}(m_\gamma),\]
with $\eta_0$ trivial and $\eta_q,\;q=1,\ldots,10$ appropriate virtual representations of $M$.

\end{proof}

%

\section{Proof of Theorems \ref{t:Length2Laplace}, \ref{t:Length2Length} and \ref{t:Length2LengthOdd}}\label{s:Length}
We can now use the alternating trace formula to relate the length spectrum with some combination of the representation spectrum.
The remarkable point is that we can actually retrieve each one of the multiplicities $\m_{\Gamma,\eta_q}$ appearing in the formula (rather than just their combination). The key to retrieving the individual multiplicities is the fact that when dilating the functions $\hat{g}_q(T\nu)$ they grow exponentially in $T$ and different values of $q$ correspond to different rates of exponential growth.
Theorems \ref{t:Length2Laplace}, \ref{t:Length2Length} and \ref{t:Length2LengthOdd}  will all follow from the following result.
\begin{thm}\label{t:Length2Spec}
Let $G$ be as above and $\Gamma_1,\Gamma_2\subseteq G$ two uniform torsion free lattices.
Assume that $\DL(\Gamma_1,\Gamma_2;T)$ satisfies
\begin{equation}\label{e:lengthbound}
\DL(\Gamma_1,\Gamma_2;T)\lesssim e^{\alpha T},
\end{equation}
for some $\alpha\in[0,\rho)$. Then $\m_{\Gamma_1,\eta_q}=\m_{\Gamma_2,\eta_q}$ for all $q<\rho-\alpha$.
\end{thm}
\begin{proof}
We retain the notation $m=\rho-\alpha_0$ with $\alpha_0\in\{0,\tfrac{1}{2},1\}$ (depending on $G$). We note that the weight function $\psi$ in the alternating trace formula satisfies $\psi(\ell)\asymp e^{\alpha_0\ell}$ for large $\ell>0$.

Assume by contradiction that $\m_{\Gamma_1,\eta_{q}}\neq \m_{\Gamma_2,\eta_{q}}$ for some $q<\rho-\alpha$ and let $q_0$ denote smallest such $q$.
We first show that the bound \eqref{e:lengthbound} implies that an appropriate grouping together of the complementary series must cancel. To do this we consider the difference between the alternating trace formulas for $\Gamma_1$ and $\Gamma_2$,
\begin{eqnarray*}
\frac{1}{2}\sum_{\ell\in \pls} \ell \Delta \m^o(\ell)\sum_{j=1}^\infty \frac{g(j \ell)}{\psi(j \ell)}&=& \sum_{q=0}^m (-1)^{q+1}\Delta V\int_\bbR \hat{g}_{q}(\nu)\mu_{\eta_q}(\nu)d\nu\\
&+&\sum_{q=0}^m(-1)^q\sum_{\nu_k\in \bbR} \Delta \m_{\eta_q}(i\nu_k)\hat{g}_{q}(\nu_k)\\
&+&\sum_{q=0}^m(-1)^q\sum_{\nu_k\in (0,\rho)} \Delta \m_{\eta_q}(\nu_k)\hat{g}_{q}(i\nu_k)\\
\end{eqnarray*}
We can rewrite the third sum as
\begin{eqnarray*}
\lefteqn{\sum_{q=0}^m(-1)^q\sum_{\nu_k\in (0,\rho)} \Delta \m_{\eta_q}(\nu_k)\hat{g}_{q}(i\nu_k)}\\
&&=\sum_{q=0}^m(-1)^q\sum_{\nu_k\in (0,\rho)} \Delta \m_{\eta_q}(\nu_k)(\hat{g}(i(\nu_k+q))+\hat{g}(i(\nu_k-q)))\\
&&=\sum_{x_k}\hat{g}(ix_k)\sum_{q=0}^m(-1)^q(\Delta \m_{\eta_q}(x_k-q)+\Delta \m_{\eta_q}(x_k-q))\\
&&=\sum_{x_k}\hat{g}(ix_k)C(x_k)\\
\end{eqnarray*}
where
\[C(x)=\sum_{q=0}^m(-1)^q(\Delta \m_{\eta_q}(x-q)+\Delta \m_{\eta_q}(x-q)).\]
Note that $C(x)\neq 0$ only on a finite (possibly empty) set which is contained in $(-m,\rho+m)$.
We show that in fact $C(x)=0$ for all $|x|>m-q_0$. Otherwise, let $x_0\in (-m,\rho+m)$ with $|x_0|>m-q_0$ be such that $C(x_0)\neq 0$ and that
$C(x)=0$ for all $|x|>|x_0|$ (we can find such a point as $C(x)\neq 0$ on a finite set).

Since we assume that $q_0<m+\alpha_0-\alpha$ and that $|x_0|>m-q_0$, we get that $\alpha-|x_0|<\alpha_0$ and we can find some $c\in [0,1)$ such that $\alpha-|x_0|<\alpha_0c<\alpha_0$ (where $c=0$ iff $\alpha_0=0$). We now estimate the two sides of the alternating trace formula with the test function
$$g_T(\ell)=\frac{x_0\hat{g}(ix_0)}{\sinh(Tx_0)}(g*\id_{cT,T})(\ell),$$
where $\id_{cT,T}$ denotes the indicator function of $[-T,-cT]\cup [cT,T]$.

Using the fact that $|g_T(\ell)|\lesssim e^{-T|x_0|}$ for $\ell\in[cT-1,T+1]$ and $g_T(\ell)=0$  for $\ell\not\in [cT-1,T+1]$ we can bound the left hand side of the trace formula by some constant times
\begin{eqnarray*}
e^{-T|x_0|}\sum_{\ell\leq T}\ell |\Delta \m^o(\ell)|\sum_{k=\frac{cT}{\ell}}^{T/\ell} e^{-\alpha_0k\ell}
&\lesssim& e^{-T(|x_0|+\alpha_0c)}\sum_{\ell\leq T} |\Delta \m^o(\ell)|\\&\lesssim &e^{-T(|x_0|+\alpha_0c-\alpha)},
\end{eqnarray*}
which goes to zero as $T\to\infty$.

For the right side of the trace formula, for any $\nu\in \bbR$  we can bound
$$|\hat{g}_{T,q}(\nu)|\lesssim Te^{T(m-q-|x_0|)}|\hat{g}_q(\nu)|$$
implying that
for all $q\geq q_0$ as $T\to\infty$
\[\sum_{\nu_k\in \bbR} |\Delta \m_{\eta_q}(i\nu_k)||\hat{g}_{T,q}(\nu_k)|\to 0.\]
If $q_0>0$ then the minimality of $q_0$ implies that $\Delta \m_{\eta_0}=0$, and since $\eta_0=1$ is the trivial representation we get that  $\Delta V=0$ (e.g., from Weyl's law).
In the case when $q_0=0$, then $m-|x_0|<0$ and the bound $|\hat{g}_{T,q}(\nu)|\lesssim Te^{T(m-|x_0|)}|\hat{g}_q(\nu)|$
implies that $\int_\bbR \hat{g}_{T,q}(\nu)\mu_q(\nu)d\nu\to 0$.
Consequently, the only part of the right hand side of the formula that does not vanish in the limit is the finite sum $\sum_{|x_k|\leq |x_0|}\hat{g}_T(ix_k)C(x_k)$
which converges to $C(x_0)$. Since the left hand side of the formula goes to zero we get that $C(x_0)=0$ in contradiction.

Next we show that $\Delta\m_{\sigma_{q_0}}(i\nu)=0$ for $\nu\in \bbR$.
To do this we fix $\nu_0\in(0,\infty)$ and estimate both sides of the alternating trace formula with the test function
$$g_{T}(\ell)= \frac{g(\frac{\ell}{T})\cos(\ell \nu_0)}{T\hat{g}(iT(m-q_0))},$$
where $g\in C^\infty(\bbR)$ is even and supported on $(-1,-c)\cup (c,1)$ with $c\in[0,1)$ satisfying that $q_0-m+\alpha<c\alpha_0<\alpha_0$ (again $c=0$ iff $\alpha_0=0$).  We note that since $g$ is supported on $(-1,-c)\cup (c,1)$ we can estimate
$$e^{c'T}\lesssim\hat{g}(iT)\lesssim e^{T}$$
for any $c<c'<1$.

For the left hand side, using \eqref{e:lengthbound} together with the bound
$$|g_{T}(\ell)|\lesssim  \frac{|g(\frac{\ell}{T})|}{T\hat{g}(iT(m-q_0))}\lesssim_{c'} \frac{e^{-T(m-q_0)c'}}{T},$$
for any $c'\in (c,1)$ and the fact that $g(\ell)=0$ for $\ell\in (-c,c)$, the same argument as above gives a bound of $O(e^{-T((m-q_0)c'+\alpha_0c-\alpha)})$.
Since $q_0<m-\alpha+\alpha_0c$ we can choose $c'$ sufficiently close to one to insure that the left hand side goes to zero as $T\to\infty$.

Next we estimate the right hand side. For $x\in (-m,\rho+m)$, if $|x|<m-q_0$, then
\[|\hat{g}_{T}(ix)|=\frac{\hat{g}(T(ix+\nu_0))+\hat{g}(T(ix-\nu_0))}{\hat{g}(iT(m-q_0))}\lesssim e^{T(|x|-c'(m-q_0))}\to 0,\]
and if $|x|=m-q_0$ then
\[|\hat{g}_{T}(ix)|=\frac{\hat{g}(T(i(m-q_0)+\nu_0))+\hat{g}(T(i(m-q)-\nu_0))}{\hat{g}(iT(m-q_0))}\lesssim \frac{1}{T\nu_0}\to 0.\]
Since we already showed that $C(x)=0$ for $|x|>m-q_0$ we get that the contribution of the complementary series
$\sum \hat{g}_{T}(ix_k) C(x_k)\to 0$ as $T\to\infty$.

For $\nu\in \bbR$ we can estimate $|\hat{g}_{T,q}(\nu)|\lesssim_N \frac{e^{-T(q-q_0)}}{1+||\nu|-\nu_0|^N}$ for all $q\geq q_0$, while if $\nu=\nu_0$ then $\hat{g}_{T,q}(\nu_0)\to 1$ as $T\to\infty$. Consequently, the right hand side of the trace formula converges to $(-1)^{q_0}\Delta \m_{\eta_{q_0}}(i\nu_0)$, and comparing to the left hand side we conclude that $\Delta \m_{\eta_{q_0}}(i\nu_0)=0$.

We have thus shown that $\Delta\m_{\sigma_{q_0}}(i\nu)=0$ for all $\nu\in \bbR\setminus \{0\}$. Consequently, Theorem \ref{t:Spec2Length2Spec} implies that $\Delta L_{\eta_{q_0}}=0$ and hence $\Delta\m_{\eta_{q_0}}(\nu)=0$ also for $\nu\in [0,\rho)$.
\end{proof}

\subsection{Proof of Theorem \ref{t:Length2Laplace}}
Assume that
$\DL(\Gamma_1,\Gamma_2;T)\lesssim e^{\alpha T}$,
with $0\leq \alpha<\rho$. Then Theorem \ref{t:Length2Spec} implies that  $\m_{\Gamma_1,\eta_0}=\m_{\Gamma_2,\eta_0}$. But $\eta_0$ is the trivial representation of $M$ and hence
$\m_{\Gamma_1}(\pi_{1,\nu})=\m_{\Gamma_2}(\pi_{1,\nu})$ for all $\nu\in (0,\rho)\cup i\bbR$.

\subsection{Proof of Theorem \ref{t:Length2Length}}
When $G/K$ is not an odd dimensional real hyperbolic space the threshold $\alpha_0>0$. Assume that
\begin{equation*}
\sum_{\ell\leq T}|\Delta L^o(\ell)|\lesssim e^{\alpha T},
\end{equation*}
with $0<\alpha<\alpha_0$. Then Theorem \ref{t:Length2Spec} implies that $\m_{\Gamma_1,\eta_q}=\m_{\Gamma_2,\eta_q}$ for
$q=0,\ldots,m$. Since these are all the representation appearing in the alternating trace formula we get that for any even test function $g\in C^\infty_c(\bbR)$
\begin{eqnarray*}
\sum_{\ell \in \pls} \ell \Delta \m^o(\ell)\sum_{j=1}^\infty \frac{g(j \ell)}{\psi(j \ell)}=0.
\end{eqnarray*}
But this can only happen if $\Delta \m^o(\ell)=0$ for all $\ell\in \bbR$ (otherwise, taking $g$ supported around the smallest $\ell$ where $\Delta \m^o(\ell)\neq 0$ will give a contradiction).
\subsection{Proof of Theorem \ref{t:Length2LengthOdd}}
Let $G=\SO_0(2m+1,1)$ and $\Gamma_1,\Gamma_2\subseteq G$ two uniform torsion free lattices.
Assume that $\m^o_{\Gamma_1}(\ell)=\m^o_{\Gamma_2}(\ell)$ for all $\ell\in \bbR$ except perhaps a finite set $\{\ell_1,\ldots,\ell_k\}$.
We thus have that $\DL(\Gamma_1,\Gamma_2;T)=O(1)$ and Theorem \ref{t:Length2Spec} implies that $\Delta \m_{\sigma_q}(\nu)=0$ for $0\leq q\leq m-1$.
This also implies that $\vol(\Gamma_1\bs G)=\vol(\Gamma_2\bs G)$ and hence the alternating trace formula takes the form
\begin{eqnarray}\label{e:remainder}
\frac{1}{2}\sum_{j=1}^k \ell_j\Delta \m^o(\ell_j)\sum_{q=1}^\infty g(q\ell_j)=(-1)^m\sum_{\nu\in S_{\sigma_m}}\Delta \m_{\sigma_m}(\nu)\hat{g}(i\nu)
\end{eqnarray}
Using Poisson summation, for each fixed $\ell_j$ we can substitute
\begin{eqnarray*}
\sum_{q=1}^\infty g(q\ell_j)&=&\frac{1}{2}\big(\sum_{q\in \bbZ}g(q\ell_j)-g(0)\big)\\
&=&\frac{1}{2}(\sum_{q\in \bbZ}\frac{\hat{g}(2\pi q/\ell_j)}{\ell_j}-g(0)\big)\\
&=&\sum_{q=0}^\infty\frac{\hat{g}(2\pi q/\ell_j)}{\ell_j}-\frac{g(0)}{2}\\
\end{eqnarray*}
Plugging this back into \eqref{e:remainder} replacing $g(0)=2\int_0^\infty \hat{g}(\nu)d\nu$ we get
\begin{eqnarray*}
\frac{1}{2}\sum_{j=1}^k\sum_{q=0}^\infty \Delta \m^o(\ell_j)\hat{g}(2\pi q/\ell_j)-\frac{1}{2}\left(\sum_{j=1}^k \ell_j\Delta \m^o(\ell_j)\right)\int_{0}^\infty\hat{g}(\nu)d\nu\\
=(-1)^m\sum_{\nu\in S_{\sigma_m}}\Delta \m_{\sigma_m}(\nu)\hat{g}(i\nu).
\end{eqnarray*}
This holds for any even holomorphic function $\hat{g}$ of uniform exponential type. Since
these test functions are dense in $C[0,\infty)$ we get an equality of the corresponding measures
\begin{eqnarray*}
\sum_{j=1}^k\sum_{q=0}^\infty \Delta \m^o(\ell_j)\delta(x-2\pi q/\ell_j)-\left(\sum_{j=1}^k \ell_j\Delta \m^o(\ell_j)\right)d\nu\\
=(-1)^m\sum_{\nu\in S_{\sigma_m}}2\Delta \m_{\sigma_m}(\nu)\delta(x-i\nu).
\end{eqnarray*}
Since there is no continuous measure on the right hand side we must have that $\sum_{j=1}^k \ell_j\Delta \m^o(\ell_j)=0$ so that $\{\ell_1,\ldots,\ell_k\}$ are  indeed linearly dependent over $\bbQ$.

Moreover, this formula gives additional information on the spectral parameters having different multiplicities. Specifically, we get that $\Delta \m_{\sigma_m}(\nu)\neq 0$ only when $\nu\in i\bbR^+$ satisfies $e^{\nu \ell_j}=1$ for some $j=1,\ldots, k$ (that is, when $i\nu=\frac{2\pi q}{\ell_j}$) in which case
\[2\Delta \m_{\sigma_m}(\nu)=(-1)^m\mathop{\sum_{j=1}^k}_{e^{\nu\ell_j}=1} \Delta \m^o(\ell_j).\]

\section{Spaces with similar length spectrum}\label{s:Similar}
In order to prove Theorem \ref{t:NoPositiveDensity} (which excludes the possibility of a positive density threshold in Theorem \ref{t:Length2Laplace})
we need to find pairs of lattices for which the length spectrum is as similar as possible. To do this, we borrow the examples constructed by Leininger, McReynolds, Neumann, and Reid \cite{LeiningerMcReynoldsNeumannReid07} of non iso-spectral manifolds having the same length sets. Their construction is based on a modification of the Sunada method \cite{Sunada85}, for constructing non isometric manifolds having the same length spectrum. We start by recalling the Sunada method and in particular his formula for multiplicities of the length spectrum of a finite cover.

\subsection{Splitting of primitive geodesics in covers}
Let $\calM=\Gamma\bs \calH$ be as above, let $\Gamma_0\subseteq \Gamma$ denote a finite index subgroup and let $\calM_0\to \calM$ the corresponding finite cover. We recall the analogy between the splitting of prime ideals in number field extensions and splitting of primitive geodesics in finite covers.

Let $\frakp$ denote a primitive conjugacy class in $\Gamma$ (corresponding to a primitive geodesic in $\calM$). We say that a primitive conjugacy class $\frakP$ in $\Gamma_0$ lies above $\frakp$ (denoted by $\frakP|\frakp$) if there is some $\gamma\in \frakp$ and a natural number $f$ such that $\frakq=[\gamma^f]_{\Gamma_0}$. We call $f$ the degree of $\frakP$ and denote $f=\deg(\frakP)$ (this does not depend on the choice of representative $\gamma\in \frakp$). We note that if $\frakP_1,\ldots,\frakP_r$ are all the primitive classes above $\frakp$, then $\sum_{j=1}^r f_j=[\Gamma:\Gamma_0]$.

Assume further that $\Gamma_0$ is normal in $\Gamma$. Let $\frakP_1,\ldots,\frakP_r$ denote all the primitive conjugacy classes in $\Gamma_0$ lying above $\frakp$. Then the group $A=\Gamma/\Gamma_0$ acts naturally on  $\{\frakP_1,\ldots,\frakP_r\}$ by permutation.
To any class $\frakP|\frakp$ we attache an element $\sigma_{\frakP}\in A$ which generates the stabilizer of $\frakP$ in $A$ (specifically, if $\frakP=[\gamma^f]_{\Gamma_0}$ for some $\gamma\in \frakp$ then $\sigma_{\frakP}$ is the class of $\gamma$ in $A=\Gamma/\Gamma_0$). We note that the conjugacy class of $\sigma_{\frakP}$ in $A$ depends only on $\frakp$ and we denote it by $\sigma_{\frakp}$.

The analogy with splitting of prime ideals in extensions of number fields is evident from our choice of notation: primitive classes correspond to prime ideals, (normal) covers to (normal) field extensions, and the element $\sigma_{\frakP}$ and class $\sigma_\frakp$ to the Frobenius element and Frobenius class in the Galois group.

\subsection{The Sunada construction}\label{s:Sunada}
We recall the general setup for the Sunada construction. Let $\calM=\Gamma\bs \calH$ as above, let
$\Gamma_0\subset\Gamma$ denote a finite index normal subgroup and let $A=\Gamma/\Gamma_0$. Let $B\subset A$ denote a subgroup and let $\Gamma_0\subset\Gamma_B\subset\Gamma$ such that $\Gamma_B/\Gamma_0=B$.
Fix a primitive class $\frakp$ in $\Gamma$ and a class $\frakP|\frakp$ in $\Gamma_0$. Let $\frakq_1,\ldots,\frakq_r$ denote all the classes in $\Gamma_B$ above $\frakp$ and let $f_j=\deg(\frakq_j)$.

The length spectrum for $\Gamma_B$ can be recovered from information on $\Gamma$ and the data on the splitting degrees as follows:
\begin{equation}\label{e:multiplicities}\m^o_{\Gamma_B}(\ell)=\sum_{d=1}^n\sum_{\ell_{\frakp}=\ell/d}\calA_B(\frakp,d),\end{equation}
where $n=[\Gamma:\Gamma_0]$, the second sum is over all primitive classes in $\Gamma$ of length $\ell/d$
and
\begin{equation}\label{e:splitting}\calA_B(\frakp,d)=\#\{\frakq: \deg(\frakq)=d \mbox{ and } \frakq|\frakp\},\end{equation}
encodes the splitting type of $\frakp$.

The information on the splitting type encoded in $\calA_B(\frakp,d)$ can be recovered from information on the finite groups $A$ and $B$.
To do this, let $\chi_B$ denote the character of the representation of $A$ given by the action of $A$ on co-sets in $A/B$, that is \begin{equation}\label{e:chiB}\chi_B(a)=\frac{|Z_A(a)||([a]\cap B)|}{|B|}=\frac{|[a]\cap B|}{|[a]\cap A|}[A:B],\end{equation}
where $[a]$ denotes the conjugacy class of $a$ in $A$.
We then have that (see \cite[Proof of Theorem 2]{Sunada85})
\[\chi_B(\sigma_\frakp^m)=\sum_{f_j|m} f_j=\sum_{d|m}d\calA_B(\frakp,d),\]
and using M\"{o}bius inversion we get
\begin{equation}\label{e:Mobius}
\calA_B(\frakp,m)=\frac{1}{m}\sum_{d|m}\mu(m/d)\chi_B(\sigma_\frakp^d),
\end{equation}
where $\mu(m)$ is the M\"{o}bius function.

Since the value of $\chi_B(a)$ is determined by $|[a]\cap B|$, we see that if $B_1,B_2\subseteq A$ satisfy that
$|[a]\cap B_1|=|[a]\cap B_2|$ for all $a\in A$ then $\Gamma_{B_1}$ and $\Gamma_{B_2}$ are length equivalent.
In fact, this condition implies that the two spaces are also representation equivalent (see \cite{Sunada85}).

\subsection{Lattices with the same length sets}
Leininger McReynolds Neumann and Reid showed in \cite{LeiningerMcReynoldsNeumannReid07} that, in the same setup, if we assume the weaker condition that
$|[a]\cap B_1|\neq 0$ if and only if $|[a]\cap B_2|\neq 0$ for any $a\in A$, then $\Gamma_{B_1}$ and $\Gamma_{B_2}$ have the same length sets. Moreover, under additional assumptions on the groups $A,B_1,B_2$ (see \cite[Definition 2.2]{LeiningerMcReynoldsNeumannReid07}) one can guarantee that the primitive length sets are the same. In the same paper they constructed explicit examples of triples $B_1, B_2\subseteq A$ satisfying these conditions and showed that these groups can be realized as quotients of lattices in any of the rank one groups.

We briefly recall their examples. Let  $p\in \bbZ$ denote a prime number and assume that it is inert in the imaginary quadratic extension $K=\bbQ(i)$. In this case we have that the ideal $(p)=p\calO$ is a prime ideal in $\calO=\bbZ[i]$ and the quotient $\calO/p\calO=\bbF_q$ is the finite field with $q=p^2$ elements. We then have the following inclusion of exact sequences
\[\begin{array}{ccccccc}
1 \to  &V_p &\to &\PSL_2(\bbZ/p^2\bbZ)    & \to & \PSL_2(\bbF_p)&\to  1 \\
      &\downarrow     &    &  \downarrow            &     & \downarrow    &       \\
1 \to  &V_{q} &\to &\PSL_2(\calO/p^2\calO)& \to & \PSL_2(\bbF_q)&\to  1
\end{array},\]
where $V_p\cong\Sl_2(\bbF_p)=\{X\in \Mat_2(\bbF_p)|\Tr(X)=0\}$  via
$$\Sl_2(\bbF_p)\ni X\mapsto I+pX\in V_p,$$ and
and similarly $V_q\cong \Sl_2(\bbF_q)$. Inside $\Sl_2(\bbF_p)\cong \bbF_p^3$ we have the plane
$$R^\bot(\bbF_p)=\{\smatrix{x}{y}{-y}{-x}\in \Sl_2(\bbF_p)|x,y\in \bbF_p\}.$$
As shown in \cite{LeiningerMcReynoldsNeumannReid07}, in any rank one group $G$ there is a uniform lattice $\Gamma$ that surjects onto $A=\PSL_2(\calO/p^2\calO)$ for infinitely many values of $p$.
In this case, the lattices $\Gamma_{B_1}$ and $\Gamma_{B_2}$, corresponding to $B_1=V_p$ and $B_2=\{I+pX|X\in R^\bot(\bbF_p)\}$, have the same primitive length sets. We now take a closer look at this example and analyze the difference between the multiplicities.

\subsection{Computing the splitting types}
Let $\Gamma$ denote a uniform lattice in a rank one group $G$ and assume that there is a surjection
$$\vphi:\Gamma\twoheadrightarrow A=\SL_2(\calO/p^2\calO).$$
Consider the following finite index subgroup: $\Gamma_0=\ker(\vphi),\; \Gamma_{V_q}=\vphi^{-1}(V_{q})$ and
$\Gamma_j=\vphi^{-1}(B_j),\; j=1,2$ where $B_1=V_p=\{I+pX|X\in \Sl_2(\bbF_p)\}$ and $B_2=\{I+pX|X\in R^\bot(\bbF_p)\}$ as above.
We note that $\Gamma_{V_q}$ and $\Gamma_0$ are both normal subgroups of $\Gamma$.

To get hold of the splitting types of primitive classes $\frakp$ in $\Gamma$ we need to count how many elements lie in the intersection of each conjugacy class with each of the above subgroups.
Since $V_q$ is normal in $\PSL_2(\calO/p^2\calO)$ a conjugacy class intersects $V_q$ if and only if it is contained there.
Furthermore, since $V_q$ is commutative, the action of $\PSL_2(\calO/p^2\calO)$ on $V_q$ by conjugation factors through the quotient $\PSL_2(\bbF_q)$. The following lemma describes the $\PSL_2(\bbF_q)$ conjugacy classes in $V_q=\Sl_2(\bbF_q)$.
\begin{lem}{\cite[Lemma 4.2]{LeiningerMcReynoldsNeumannReid07}}\label{l:VP}
For each $Q\in \bbF_q^*$ the set
$$\calC_Q(\bbF_q)=\{X\in \Sl_2(\bbF_q)|\det(X)=-Q\}$$
is a conjugacy class. In addition to these classes and the trivial class there are two nilpotent classes with representatives $\left(\begin{smallmatrix} 0 & 0 \\ 1& 0\end{smallmatrix}\right)$ and $\left(\begin{smallmatrix} 0 & 0 \\ \eps & 0\end{smallmatrix}\right)$ with $\eps\in \bbF_q^*$ not a square. Moreover, for any $X\in \Sl_2(\bbF_q)$ the size of its conjugacy class is given by
\[|[X]|=\left\{\begin{array}{ll}
1& X=0\\
q(q+1) & X\in \calC_Q,\; Q\in (\bbF_q^*)^2\\
q(q-1) & X\in \calC_Q,\; Q\in \bbF_q^*\setminus (\bbF_q^*)^2\\
\frac{q^2-1}{2} & X \mbox{ is nilpotent.}
\end{array}\right.\]
\end{lem}
We now compute the size of the intersection of each of these classes with $V_p\cong \Sl_2(\bbF_p)$ and $R^\bot(\bbF_p)$.

\begin{lem}\label{l:Vp}
For any $X\in \Sl_2(\bbF_q)$ we have
\[|[X]\cap \Sl_2(\bbF_p)|=\left\lbrace\begin{array}{ll}
0 & X\in \calC_Q(\bbF_q),\; Q\in \bbF_q^*\setminus \bbF_p^*\\
1 & X=0\\
p(p+1) &  X\in \calC_Q(\bbF_q),\; Q\in (\bbF_p^*)^2\\
p(p-1) & X\in \calC_Q(\bbF_q),\; Q\in \bbF_p^*\setminus (\bbF_p^*)^2\\
\frac{p^2-1}{2} & X \mbox{ is nilpotent}\\
\end{array}\right.\]
\[|[X]\cap R^\bot(\bbF_p)|=\left\lbrace\begin{array}{ll}
0 & X\in \calC_Q(\bbF_q),\; Q\in \bbF_q^*\setminus \bbF_p^*\\
1 & X=0\\
p-1 & otherwise
\end{array}\right.\]
\end{lem}
\begin{proof}
The first part follows from the previous lemma together with the observation that
$\calC_Q(\bbF_q)\cap \Sl_2(\bbF_p)=\calC_Q(\bbF_p)$ and that the intersection of a nilpotent conjugacy class with $\Sl_2(\bbF_p)$ is the corresponding $\SL_2(\bbF_p)$ conjugacy class.

For the second part, note that
$$\calC_Q(\bbF_q)\cap R^\bot(\bbF_p)=\{\smatrix{x}{y}{-y}{-x}|x,y\in \bbF_p, x^2-y^2=Q\},$$
which is of size $p-1$ for any $Q\in \bbF_p^*$. The only nontrivial elements in $R^\bot$ not yet accounted for are the ones with zero determinant. There are $2(p-1)$ such elements and since $X$ is conjugated to $\left(\begin{smallmatrix} 0 & 0 \\ 1 & 0\end{smallmatrix}\right)$ if and only if $\eps X$ is conjugated to $\left(\begin{smallmatrix} 0 & 0 \\ \eps & 0\end{smallmatrix}\right)$ we have exactly $p-1$ elements in each of the remaining two conjugacy classes.
\end{proof}

Fix a primitive conjugacy class $\frakp$ in $\Gamma$, a primitive class $\frakP$ in $\Gamma_0$ above $\frakp$, and let $\sigma_\frakP\in A$ denote the ``Frobenius" element. (That is, $\sigma_{\frakP}=\vphi(\gamma)$ for $\gamma\in\frakp$ satisfying that $[\gamma^{\deg(\frakP)}]_{\Gamma_0}=\frakP$).
Since $\Gamma_{V_q}$ is normal in $\Gamma$, there is a natural number $d_0=d_0(\frakp)$ such that any $\gamma\in \frakp$ satisfies that $\gamma^{d_0}$ is primitive in $\Gamma_{V_q}$. Consequently, we have that $\sigma_\frakP^{d}\in V_q$ if and only if $d_0|d$ and that $\calA_{V_q}(\frakp,d)=\frac{|A|}{|V_q|}\delta_{d,d_0}$. We now compute the splitting types for $B_1$ and $B_2$.
\begin{lem}\label{l:splitting}
For $j=1,2$. If $\sigma_\frakp^{d_0}=1$ then
$\calA_{B_j}(\frakp,d)=\frac{|A|}{d_0|B_j|}\delta_{d,d_0}$.
Otherwise
\begin{eqnarray*}
\calA_{B_j}(\frakp,d)=\frac{\chi_{B_j}(\sigma_{\frakp}^{d_0})}{d_0} \delta_{d,d_0} +\frac{(\frac{|A|}{|B_j|}-\chi_{B_j}(\sigma_{\frakp}^{d_0}))}{pd_0} \delta_{d,pd_0}
.\\
\end{eqnarray*}
\end{lem}
\begin{proof}
Since $\sigma_{\frakp}^d\in V_q$ if and only if $d_0|d$ we get that $\chi_{B_j}(\sigma_\frakp^d)=0$ unless $d=kd_0$ for some $k\in \bbN$, and hence $\calA_{B_j}(\frakp,d)=0$ unless $d=md_0$ for some $m\in \bbN$.

If $\sigma^{d_0}=1$ then $\sigma_{\frakp}^{kd_0}=1$ for all $k\in \bbN$, and hence $\chi_{B_j}(\sigma_\frakp^{kd_0})=\frac{|A|}{|B_j|}$ for all $k\in \bbN$. Using \eqref{e:Mobius} we get that in this case
\[\calA_{B_j}(\frakp,md_0)=\frac{1}{md_0}\sum_{k|m}\mu(\frac{m}{k})\chi_{B_j}(\sigma_\frakp^{kd_0})=\frac{|A|}{d_0|B_j|}\delta_{m,1}.\]

When $\sigma_\frakP^{d_0}\neq 1$ we write $\sigma_{\frakp}^{d_0}=I+pX_0$ with $X_0\in \Sl_2(\bbF_q)$ nontrivial.
In this case, $\sigma_{\frakp}^{kd_0}=(I+pX_0)^k=I+pkX_0$. In particular, if $p|k$ then $\sigma_{\frakp}^{kd_0}=1$ and again $\chi_{B_j}(\sigma_\frakp^{kd_0})=\frac{|A|}{|B_j|}$.
For $k$ prime to $p$, if $X_0\in \calC_Q$ then $kX_0\in \calC_{k^2Q}$ and if $X\sim \smatrix00c0$ then $kX\sim \smatrix00{kc}0$. Thus, the sizes of the conjugacy classes of $[kX_0]$ and $[kX_0]\cap B_j$ don't depend on $k$ so
$\chi_{B_j}(\sigma_\frakp^{kd_0})=\chi_{B_j}(\sigma_\frakp^{d_0})$.
Using \eqref{e:Mobius} we get
\begin{eqnarray*}
\calA_{B_j}(\frakp,d_0m)&=&\frac{1}{d_0m}\sum_{k|m}\mu(\frac{m}{k})\chi_{B_j}(\sigma_{\frakp}^{kd_0})\\
&=&\frac{\chi_{B_j}(\sigma_{\frakp}^{d_0})}{d_0m}\mathop{\sum_{k|m}}_{(k,p)=1}\mu(\frac{m}{k})+
\frac{|A|}{|B_j|d_0m}\mathop{\sum_{k|m}}_{p|k}\mu(\frac{m}{k}).
\end{eqnarray*}
When $(m,p)=1$, the second sum is empty and the first sum is $\sum_{k|m}\mu(\frac{m}{k})=\delta_{m,1}$.  When $m=p\tilde{m}$ with $(\tilde{m},p)=1$, the second sum is $\delta_{\tilde{m},1}$ and the first sum is $$\mathop{\sum_{k|p\tilde{m}}}_{(k,p)=1}\mu(\frac{m}{k})=-\sum_{k|\tilde{m}}\mu(\frac{\tilde{m}}{k})=-\delta_{\tilde{m},1}.$$ Finally, when $m=p^k\tilde{m}$ with $k>1$, both sums vanish. We thus get that in any case where $X_0\neq 0$ we have
\begin{eqnarray*}
\calA_{B_j}(\frakp,d_0m)=\frac{\delta_{m,1}}{d_0}\chi_{B_j}(\sigma_{\frakp}^{d_0})+\frac{\delta_{m,p}}{pd_0}
(\frac{|A|}{|B_j|}-\chi_{B_j}(\sigma_{\frakp}^{d_0})).\\
\end{eqnarray*}
\end{proof}
\subsection{Proof of Theorem \ref{t:NoPositiveDensity}}

Theorem \ref{t:NoPositiveDensity} follows from the following proposition with $p$ a sufficiently large inert prime.
\begin{prop}
Let $\Gamma_2\subseteq \Gamma_1\subseteq \Gamma$ be as above. Then there is $\alpha<2\rho$ such that
$$\DL(\Gamma_1,\Gamma_2;T)\leq \frac{4\mathrm{Li}(e^{2\rho T})}{p+1}+O(e^{\alpha T}).$$
\end{prop}
\begin{proof}
We separate the primitive classes in $\Gamma$ into $5$ types as follows. Let $d_0=d_0(\frakp)$ as above and write $\sigma_{\frakP}^{d_0}=I+pX_0$ with $X_0\in \Sl_2(\bbF_q)$. We then say that
$\frakp$ is of \textit{trivial type} if $X_0=0$, and of \textit{nilpotent type} if $X_0$ is nilpotent, and we say that $\frakp$ is of \textit{irregular type}, \textit{quadratic type}, or \textit{non-quadratic type} if $X_0\in \calC_Q(\bbF_q)$ with $Q\in \bbF_q^*\setminus \bbF_p^*,\;Q\in (\bbF_p^*)^2$ or $Q\in \bbF_p^*\setminus (\bbF_p^*)^2$ respectively (recall that $\bbF^*_{p}\subset (\bbF_{q}^*)^2$).
Note that even though $X_0$ depends on the choice of $\frakP$ above $\frakp$, this characterization depends only on $\frakp$.

Consider the partial counting functions $\m_{\rm{tr}},\;\m_{\rm{ir}},\; \m_{\rm{qr}},\;\m_{\rm{nq}}$ and $\m_{\rm{ni}}$ each given by
\[\m_I(\ell,d)=\#\{\frakp \mbox{ of type } I| d_0(\frakp)=d \mbox{ and } \ell_\frakp=\ell/d\}.\]
where $I$ is one of the types: trivial, irregular, quadratic, non-quadratic, and nilpotent.
For the non-trivial types we define additional counting functions
\[\tilde{\m}_I(\ell,d)=\#\{\frakp \mbox{ of type } I| pd_0(\frakp)=d \mbox{ and } \ell_\frakp=\ell/d\}.\]
Let $n=|A/B_1|=p^3|\PSL_2(\bbF_p)|=\frac{p^4(p^2-1)}{2}$.
Using Lemma \ref{l:Vp} we can compute $\chi_{B_j}(\sigma_\frakp^{d_0})$ for each type of $\frakp$ and using Lemma \ref{l:splitting} we get that
\[\calA_{B_1}(\frakp,d)=\frac{n}{d_0}\left\lbrace\begin{array}{ll}
\delta_{d,d_0} & \frakp \mbox{ is trivial}\\
&\\
\frac{1}{p}\delta_{d,pd_0} & \frakp \mbox{ is irregular}\\
&\\
\frac{p+1}{p^3+p}\delta_{d,d_0}+\frac{p^3-1}{p^4+p^2}\delta_{d,pd_0} & \frakp \mbox{ is quadratic}\\
&\\
\frac{p-1}{p^3+p}\delta_{d,d_0}+\frac{p^3+1}{p^4+p^2}\delta_{d,pd_0} & \frakp \mbox{ is non-quadraric}\\
&\\
\frac{p^2-1}{p^4-1}\delta_{d,d_0}+\frac{p^3-p}{p^4-1}\delta_{d,pd_0} & \frakp \mbox{ is nilpotent}\\
\end{array}\right.\]
and
\[\calA_{B_2}(\frakp,d)=\frac{n}{d_0}\left\lbrace\begin{array}{ll}
p\delta_{d,d_0} & \frakp \mbox{ is trivial}\\
&\\
\delta_{d,pd_0} & \frakp \mbox{ is irregular}\\
&\\
\frac{p-1}{p^3+p}\delta_{d,d_0}+\frac{p^4+p^2-p+1}{p^4+p^2}\delta_{d,pd_0} & \frakp \mbox{ is quadratic}\\
&\\
\frac{p-1}{p^3+p}\delta_{d,d_0}+\frac{p^4+p^2-p+1}{p^4+p^2}\delta_{d,pd_0}  & \frakp \mbox{ is non-quadraric}\\
&\\
\frac{p^3-p}{p^4-1}\delta_{d,d_0}+\frac{p^4-2p+1}{p^4-1}\delta_{d,pd_0} & \frakp \mbox{ is nilpotent}\\
\end{array}\right.\]

Using \eqref{e:multiplicities} for $\m_{\Gamma_j}^o(\ell)$ and summing separately over classes of each type we get that
\begin{eqnarray*}
\lefteqn{\m^o_{\Gamma_1}(\ell)=\sum_{d|n}\frac{n}{d}\bigg( \m_{\rm{tr}}(\ell,d)+\tfrac{p+1}{p^3+p}\m_{\rm{qr}}(\ell,d)}\\
&&+\tfrac{p-1}{p^3+p}\m_{\rm{nq}}(\ell,d)+\tfrac{p^2-1}{p^4-1}\m_{\rm{ni}}(\ell,d)+\tfrac{1}{p}\tilde{\m}_{\rm{ir}}(\ell,d)\\
&&+\tfrac{p^3-1}{p^4+p^2}\tilde{\m}_{\rm{qr}}(\ell,d)+\tfrac{p^3+1}{p^4+p^2}\tilde{\m}_{\rm{nq}}(\ell,d)+\tfrac{p^3-p}{p^4-1}\tilde{\m}_{\rm{ni}}(\ell,d)\bigg),
\end{eqnarray*}
and
\begin{eqnarray*}
\lefteqn{\m^o_{\Gamma_2}(\ell)=\sum_{d|n}\frac{n}{d}\bigg( p\m_{\rm{tr}}(\ell,d)+\tfrac{p-1}{p^3+p}\m_{\rm{qr}}(\ell,d)+\tfrac{p-1}{p^3+p}\m_{\rm{nq}}(\ell,d)}\\
&&+\tfrac{p^3-p}{p^4-1}\m_{\rm{ni}}(\ell,d)+\tilde{\m}_{\rm{ir}}(\ell,d)+\tfrac{p^4+p^2-p+1}{p^4+p^2}\tilde{\m}_{\rm{qr}}(\ell,d)\\
&&+
\tfrac{p^4+p^2-p+1}{p^4+p^2}\tilde{\m}_{\rm{nq}}(\ell,d)+\tfrac{p^4-2p+1}{p^4-1}\tilde{\m}_{\rm{ni}}(\ell,d)\bigg).
\end{eqnarray*}
From these formulas one can directly see that, indeed, $\Gamma_1$ and $\Gamma_2$ have the same primitive length sets. Moreover, we get the following formula for the difference of the multiplicities
\begin{eqnarray*}
\m^o_{\Gamma_2}(\ell)-\m^o_{\Gamma_2}(\ell)&=&\sum_{d|n}\frac{n}{d}\bigg( (p-1)\m_{\rm{tr}}(\ell,d)-\tfrac{2}{p^3+p}\m_{\rm{qr}}(\ell,d)\\ \nonumber
&&+\tfrac{p^3-p^2-p+1}{p^4-1}\m_{\rm{ni}}(\ell,d)+\tfrac{p-1}{p}\tilde{\m}_{\rm{ir}}(\ell,d)\\ \nonumber
&&+\tfrac{p^4-p^3+p^2-p+2}{p^4+p^2}\tilde{\m}_{\rm{qr}}(\ell,d)+\tfrac{p^4-p^3+p^2-p}{p^4+p^2}\tilde{\m}_{\rm{nq}}(\ell,d).\\ \nonumber
&&+\tfrac{p^4-p^3-p+1}{p^4-1}\tilde{\m}_{\rm{ni}}(\ell,d)\bigg).
\end{eqnarray*}

We can also separate the sum over all classes in $\Gamma$ of length$\leq T$ into the different types; consequently, if we define $\pi_I(T)$ and $\tilde{\pi}_I(T)$ respectively by
$$\pi_I(T)=\sum_{\ell\leq T} \sum_{d|n}\frac{n}{d} \m_I(\ell,d) ,\quad\tilde{\pi}_I(T)=\sum_{\ell\leq T} \sum_{d|n}\frac{n}{d}\tilde{\m}_I(\ell,d) ,$$
we get that
\begin{eqnarray}\label{e:explicitdifference}
\sum_{\ell\leq T}(\m^o_{\Gamma_2}(\ell)-\m^o_{\Gamma_1}(\ell))
&=&(p-1)\pi_{\rm{tr}}(T)+\tfrac{p^3-p^2-p+1}{p^4-1}\pi_{\rm{ni}}(T)\\\nonumber
&+& \tfrac{p-1}{p}\tilde{\pi}_{\rm{ir}}(T)+\tfrac{p^4-p^3+p^2-p+2}{p^4+p^2}\tilde{\pi}_{\rm{qr}}(T)\\\nonumber
&+& \tfrac{p^4-p^3+p^2-p}{p^4+p^2}\tilde{\pi}_{\rm{nq}}(T)+\tfrac{p^4-p^3-p+1}{p^4-1}\tilde{\pi}_{\rm{ni}}(T)\\\nonumber
&-& \tfrac{2}{p^3+p}\pi_{\rm{qr}}(T).
\end{eqnarray}

Form the Prime Geodesic Theorem we know that for each of the two lattices
$\sum_{\ell\leq T} \m^o_{\Gamma_j}(\ell)=\rm{Li}(e^{2\rho T})+O(e^{\alpha T})$,
with $\alpha<2\rho$.  Hence, the difference (without absolute values) satisfies
$$\sum_{\ell\leq T} (\m^o_{\Gamma_2}(\ell)-\m^o_{\Gamma_1}(\ell))=O(e^{\alpha T}).$$
This estimate together with \eqref{e:explicitdifference} implies that
\begin{eqnarray}\label{e:piqr}
\tfrac{2}{p^3+p}\pi_{\rm{qr}}(T)&=& (p-1)\pi_{\rm{tr}}(T)+\tfrac{p^3-p^2-p+1}{p^4-1}\pi_{\rm{ni}}(T)+ \tfrac{p-1}{p}\tilde{\pi}_{\rm{ir}}(T)\\ \nonumber
&+&\tfrac{p^4-p^3+p^2-p+2}{p^4+p^2}\tilde{\pi}_{\rm{qr}}(T)+\tfrac{p^4-p^3+p^2-p}{p^4+p^2}\tilde{\pi}_{\rm{nq}}(T)\\ \nonumber
&+& \tfrac{p^4-p^3-p+1}{p^4-1}\tilde{\pi}_{\rm{ni}}(T)+ O(e^{\alpha T})
\end{eqnarray}

On the other hand, we can bound
\begin{eqnarray*}
|\m^o_{\Gamma_2}(\ell)-\m^o_{\Gamma_2}(\ell)|&\leq &\sum_{d|n}\frac{n}{d}\bigg( (p-1)\m_{\rm{tr}}(\ell,d)+\tfrac{2}{p^3+p}\m_{\rm{qr}}(\ell,d)\\
&&+\tfrac{p^3-p^2-p+1}{p^4-1}\m_{\rm{ni}}(\ell,d)+\tfrac{p-1}{p}\tilde{\m}_{\rm{ir}}(\ell,d)\\
&&+\tfrac{p^4-p^3+p^2-p+2}{p^4+p^2}\tilde{\m}_{\rm{qr}}(\ell,d)+\tfrac{p^4-p^3+p^2-p}{p^4+p^2}\tilde{\m}_{\rm{nq}}(\ell,d)\\
&&+\tfrac{p^4-p^3-p+1}{p^4-1}\tilde{\m}_{\rm{ni}}(\ell,d)\bigg),
\end{eqnarray*}
implying that
\begin{eqnarray*}
\sum_{\ell\leq T}|\m^o_{\Gamma_2}(\ell)-\m^o_{\Gamma_1}(\ell)|
&\leq &(p-1)\pi_{\rm{tr}}(T)+\tfrac{p^3-p^2-p+1}{p^4-1}\pi_{\rm{ni}}(T)\\\nonumber
&+& \tfrac{p-1}{p}\tilde{\pi}_{\rm{ir}}(T)+\tfrac{p^4-p^3+p^2-p+2}{p^4+p^2}\tilde{\pi}_{\rm{qr}}(T)\\\nonumber
&+& \tfrac{p^4-p^3+p^2-p}{p^4+p^2}\tilde{\pi}_{\rm{nq}}(T)+\tfrac{p^4-p^3-p+1}{p^4-1}\tilde{\pi}_{\rm{ni}}(T)\\\nonumber
&+& \tfrac{2}{p^3+p}\pi_{\rm{qr}}(T)=\tfrac{4}{p^3+p}\pi_{\rm{qr}}(T)+O(e^{\alpha T}).
\end{eqnarray*}
where the last line follows from \eqref{e:piqr}.
Finally, since
\[\tfrac{p+1}{p^3+p}\pi_{\rm{qr}}(T)\leq \sum_{\ell\leq T} \m^o_{\Gamma_1}(\ell)= \rm{Li}(e^{2\rho T})+O(e^{\alpha T}),\]
we get that indeed
\[\sum_{\ell\leq T} |\m^o_{\Gamma_2}(\ell)-\m^o_{\Gamma_1}(\ell)|\leq \frac{4\rm{Li}(e^{2\rho T})}{p+1}+O(e^{\alpha T}).\]
\end{proof}

\subsection{Comparison to splitting of primes in number fields}\label{s:Perlis}
In the examples constructed above, the volumes of the two non iso-spectral spaces grow with $p$.
It is thus possible that a positive density threshold can hold under the additional assumption that the volume is uniformly bounded.
To illustrate this we will look at a similar situation in the analogous question regarding the splitting of prime ideals in number fields.

Let $K/\bbQ$ denote a finite field extension and $\calO_K$ the corresponding ring of integers. For any prime number $p\in \bbZ$ we have a decomposition  $p\calO_K=\frakp_1^{e_1}\cdots\frakp_r^{e_r}$  with $\frakp_j\subseteq \calO_K$ the prime ideals dividing $p$ and $e_j$ the ramification degrees (which are all one except for the ramified primes dividing the discriminant). Each quotient $\calO_K/\frakp_j$ is a finite field of order $p^{f_j}$ where $f_j$ denotes the inertia degree of $\frakp_j$. The splitting type of $p$ in $K$ is then given by the numbers $\calA_K(p,d)=\#\{j| f_j=d\}$ for $d=1,\ldots,n=[K:\bbQ]$.

Let $K_1,K_2$ denote two number fields. We say that a prime number $p\in \bbZ$ have the same splitting types in $K_1$ and $K_2$ if $\calA_{K_1}(p,d)=\calA_{K_2}(p,d)$ for all $d=1,\ldots, n$. In \cite[Theorem 1]{Perlis77}, Perlis showed that if all but finitely many primes have the same splitting types in $K_1$ and $K_2$, then all primes have the same splitting types. In this case the fields have the same Dedekind Zeta functions, but they are not (necessarily) isomorphic; such fields are called arithmetically equivalent fields.

The finite set of exceptions in Perlis's theorem can be replaced with an infinite set of density zero. A positive density threshold that is independent of the number fields probably does not hold. Nevertheless, it is possible to give a positive density threshold which depends on the degree of the smallest Galois field $L/\bbQ$ containing $K_1$ and $K_2$. Specifically, one can show
\begin{prop}
In the above setting, for any $1\leq d\leq n$ if the set $$P_{\mathrm{bad}}(d)=\{p\; \mathrm{unramified} |\calA_{K_1}(p,d)\neq \calA_{K_2}(p,d)\}$$ has Dirichlet density smaller then $1/|\Gal(L/\bbQ)|$ then $P_{\mathrm{bad}}(d)=\emptyset$. In particular, if $P_{\rm{bad}}=\cup_{d=1}^n P_{\mathrm{bad}}(d)$ has density below $1/|\Gal(L/\bbQ)|$ then $K_1$ and $K_2$ are arithmetically equivalent.
\end{prop}
\begin{proof}
The proof is a direct result of Chebotarev's Density Theorem combined with the following observation.
For any prime $p$ denote by $\sigma_p$ the Frobenius conjugacy class in $\Gal(L/\bbQ)$. Denote by $A=\Gal(L/\bbQ)$ and by $B_j=\Gal(L/K_j)$. Then (using the same argument as in section \ref{s:Sunada}) the condition $\calA_{K_1}(p,d)\neq \calA_{K_2}(p,d)$ is equivalent to the condition
\[\sum_{m|d}\mu(d/m)\chi_{B_1}(\sigma_\frakp^m)\neq \sum_{m|d}\mu(d/m)\chi_{B_2}(\sigma_\frakp^m),\]
where $\chi_{B_j}$ is, as before, the character of the permutation representation of $A$ on $A/B_j$.

Consider the set
$$S_{\rm{bad}}(d)=\{a\in A|\sum_{m|d}\mu(d/m)\chi_{B_1}(a^m)\neq \sum_{m|d}\mu(d/m)\chi_{B_2}(a^m)\}.$$
This set is invariant under conjugation in $A$ and satisfies that
$$P_{\rm{bad}}(d)=\{p \mbox{ unramified}| \sigma_p\in S_{\rm{bad}}(d)\}.$$
Now, the Chebotarev density theorem implies that $P_{\rm{bad}}(d)$ has Dirichlet density $|S(d)|/|A|$.
Thus if the density of $P_{\rm{bad}}(d)$ is smaller then $1/|A|$ then $S_{\rm{bad}}(d)=\emptyset$ implying that $P_{\rm{bad}}(d)=\emptyset$ as claimed.
\end{proof}
\begin{rem}
It follows from the proof that if $P_{\rm{bad}}(1)=\emptyset$ then also $S_{\rm{bad}}(1)=\emptyset$ implying that $\chi_{B_1}=\chi_{B_2}$ and hence that $S_{\rm{bad}}(d)=\emptyset$ for all $d$. Consequently, we see that to prove arithmetic equivalence one does not need to consider the full splitting type but only the number of prime ideals of inertia degree one dividing each prime.
\end{rem}



\begin{thebibliography}{10}
\bibitem[Au06]{Aubry06}
Jean-Marie Aubry, \emph{On the rate of pointwise divergence of {F}ourier and
  wavelet series in {$L^p$}}, J. Approx. Theory \textbf{138} (2006), no.~1,
  97--111.
\bibitem[BR10]{BhagwatRajan10}
C.~{Bhagwat} and C.~S. {Rajan}, \emph{On a spectral analogue of the strong
  multiplicity one theorem}, IMRN (2011) \texttt{doi:10.1093/imrn/rnq243}.

\bibitem[BR11]{BhagwatRajan11}
\bysame, \emph{On a multiplicity one property for the length spectra of even dimensional compact hyperbolic spaces},
to appear in J. Number Theory. Preprint \texttt{arXiv:1102.0363}

\bibitem[Bu92]{Buser92}
P.~Buser, \emph{Geometry and spectra of compact Riemann surfaces.}
Progress in Mathematics, 106. Birkhäuser Boston, Inc.,

\bibitem[Di89]{Dietmar89}
A.~Deitmar, \emph{The Selberg trace formula and the Ruelle zeta function for compact hyperbolics.}
Abh. Math. Sem. Univ. Hamburg \textbf{59} (1989), 101–106.

\bibitem[DG75]{DuistermaatGuillemin75}
J.~J. Duistermaat and V.~W. Guillemin, \emph{The spectrum of positive elliptic
  operators and periodic bicharacteristics}, Invent. Math. \textbf{29} (1975),
  no.~1, 39--79.

\bibitem[EGM]{ElstrodtGrunewaldMennicke98}
J.~Elstrodt, F.~Grunewald, and J.~Mennicke, \emph{Groups acting on hyperbolic
  space}, Springer Monographs in Mathematics, Springer-Verlag, Berlin, 1998,
  Harmonic analysis and number theory.

\bibitem[Ga77]{Gangolli77}
Ramesh Gangolli, \emph{The length spectra of some compact manifolds of negative
  curvature}, J. Differential Geom. \textbf{12} (1977), no.~3, 403--424.

\bibitem[Hu59]{Huber59}
H.~Huber \emph{Zur analytischen {T}heorie hyperbolischen {R}aumformen und
              {B}ewegungsgruppen}, Math. Ann. \textbf{138} (1959) 1--26


\bibitem[Kn86]{Knapp86}
Anthony~W. Knapp, \emph{Representation theory of semisimple groups}, Princeton
  Mathematical Series, vol.~36, Princeton University Press, Princeton, NJ,
  1986, An overview based on examples.

\bibitem[Kn02]{Knapp02}
\bysame, \emph{Lie groups beyond an introduction}, second ed., Progress in
  Mathematics, vol. 140, Birkh\"auser Boston Inc., Boston, MA, 2002.


\bibitem[La76]{Lang76}
Serge Lang, \emph{Introduction to modular forms}, Springer-Verlag, Berlin,
  1976, Grundlehren der mathematischen Wissenschaften, No. 222.

\bibitem[LMNR]{LeiningerMcReynoldsNeumannReid07}
C.~J. Leininger, D.~B. McReynolds, W.~D. Neumann, and A.~W. Reid, \emph{Length
  and eigenvalue equivalence}, Int. Math. Res. Not. IMRN (2007), no.~24,
  doi:10.1093/imrn/rnm132.

\bibitem[Ma69]{Margulis69}
G.~A. Margulis, \emph{Certain applications of ergodic theory to the
  investigation of manifolds of negative curvature}, Funkcional. Anal. i
  Prilo\v zen. \textbf{3} (1969), no.~4, 89--90.

\bibitem[Mi82]{Miatello82}
Roberto~J. Miatello, \emph{An alternating sum formula for multiplicities in
  {$L^{2}(\Gamma \backslash G)$}}, Trans. Amer. Math. Soc. \textbf{269} (1982),
  no.~2, 567--574.

\bibitem[MV83]{MiatelloVargas83}
Roberto~J. Miatello and Jorge~A. Vargas, \emph{On the distribution of the
  principal series in {$L^{2}(\Gamma \backslash G)$}}, Trans. Amer. Math. Soc.
  \textbf{279} (1983), no.~1, 63--75.


\bibitem[Pe77]{Perlis77}
R.~Perlis, \emph{On the equation $\zeta_{K}(s)=\zeta_{K'}(s)$},
J. Number Theory \textbf{9} (1977), 342–360.

\bibitem[Ra94]{Ramakrishnan94}
Dinakar Ramakrishnan, \emph{A refinement of the strong multiplicity one theorem
  for {${\rm GL}(2)$}. {A}ppendix to: ``{$l$}-adic representations associated
  to modular forms over imaginary quadratic fields. {II}''
  by {R}. {T}aylor}, Invent. Math. \textbf{116} (1994), no.~1-3, 645--649.

\bibitem[Ru66]{Rudin66}
Walter Rudin, \emph{Real and complex analysis}, McGraw-Hill Book Co., New York,
  1966.

\bibitem[Sa02]{Salvai02}
Marcos Salvai, \emph{On the {L}aplace and complex length spectra of locally
  symmetric spaces of negative curvature}, Math. Nachr. \textbf{239/240}
  (2002), 198--203.

\bibitem[SW99]{SarnakWakayama99}
Peter Sarnak and Masato Wakayama, \emph{Equidistribution of holonomy about
  closed geodesics}, Duke Math. J. \textbf{100} (1999), no.~1, 1--57.

\bibitem[So04]{Soundararajan04}
K.~Soundararajan, \emph{Strong multiplicity one for the Selberg class}, Canad. Math. Bull. \textbf{47} (2004), no. 3, 468–474.

\bibitem[Su85]{Sunada85}
Toshikazu Sunada, \emph{Riemannian coverings and isospectral manifolds}, Ann.
  of Math. (2) \textbf{121} (1985), no.~1, 169--186.

\bibitem[Vi80]{Vigneras80}
Marie-France Vign{\'e}ras, \emph{Arithm\'etique des alg\`ebres de quaternions},
  Lecture Notes in Mathematics, vol. 800, Springer, Berlin, 1980.

\bibitem[Wal]{Wallach76}
N.~Wallach, \emph{On the Selberg trace formula in the case of compact quotient}, Bull. Amer. Math. Soc. \textbf{82} (1976), 171--195.

\bibitem[War]{Warner79}
G.~Warner, \emph{Selberg's trace formula for nonuniform lattices: The $\bbR$-rank one case} in \emph{Studies in Algebra and Number Theory}, Adv. in Math. Suppl. Stud. \textbf{6}, Academic Press, New York, (1979), 1--142


%
%
%
%
%
%
%
%
%
%
%
%
%
%
%
%
%
%
%
%

\end{thebibliography}
\def\cprime{$'$} \def\cprime{$'$}
\providecommand{\bysame}{\leavevmode\hbox to3em{\hrulefill}\thinspace}
\providecommand{\MR}{\relax\ifhmode\unskip\space\fi MR }
\MRhref{is called by the amsart/book/proc definition of \MR.}
\providecommand{\MRhref}[2]{%
  \href{http://www.ams.org/mathscinet-getitem?mr=#1}{#2}
}
\providecommand{\href}[2]{#2}


\end{document}